\documentclass[12pt,reqno,tbtags,draft]{amsart}
\usepackage{amssymb}
\usepackage{ifdraft}
\usepackage{url}
\usepackage[square,numbers]{natbib}

\title{Consistent random vertex-orderings of graphs}
\date{10 June, 2015}

\author{Paul Balister}
\thanks{The first author is partially supported by NSF grant DMS~1301614.}
\address{Department of Math Sciences, University of Memphis, TN 38152, USA}
\email{pbalistr@memphis.edu,
\url{http://www.memphis.edu/msci/people/pbalistr.php}}

\author{B\'ela Bollob\'as}
\thanks{The second author is partially supported by NSF grant DMS~1301614
and MULTIPLEX no.\ 317532.}
\address{Department of Pure Mathematics and Mathematical Statistics,
Wilberforce Road, Cambridge CB3\thinspace0WB, UK; {\em and\/}
Department of Mathematical Sciences, University of Memphis, Memphis TN 38152,
USA; {\em and\/} London Institute for Mathematical Sciences, 35a South St.,
Mayfair, London W1K\thinspace2XF, UK.}
\email{bollobas@dpmms.cam.ac.uk}

\author{Svante Janson}
\thanks{The third author is partly supported by the Knut and Alice
Wallenberg Foundation.}
\address{Department of Mathematics, Uppsala University, PO Box 480,
SE-751~06 Uppsala, Sweden}
\email{svante.janson@math.uu.se, \url{http://www2.math.uu.se/~svante/}}

\subjclass[2010]{60C05, 05C60, 60G09}

\advance\textwidth1in
\bibpunct[, ]{[}{]}{;}{n}{,}{,}
\numberwithin{equation}{section}
\allowdisplaybreaks

\newtheorem{theorem}{Theorem}[section]
\newtheorem{lemma}[theorem]{Lemma}

\newtheorem{corollary}[theorem]{Corollary}
\theoremstyle{definition}
\newtheorem{example}[theorem]{Example}

\newtheorem{remark}[theorem]{Remark}

\theoremstyle{remark}

\newcommand\refT[1]{Theorem\/~\textup{\ref{t:#1}}}
\newcommand\refC[1]{Corollary\/~\textup{\ref{c:#1}}}
\newcommand\refL[1]{Lemma\/~\textup{\ref{l:#1}}}
\newcommand\refR[1]{Remark\/~\textup{\ref{r:#1}}}
\newcommand\refS[1]{Section\/~\textup{\ref{s:#1}}}

\newcommand\refE[1]{Example\/~\textup{\ref{x:#1}}}
\newcommand\refF[1]{Figure\/~\textup{\ref{f:#1}}}

\newenvironment{romenumerate}[1][-10pt]{
\addtolength{\leftmargini}{#1}\begin{enumerate}
}{\end{enumerate}}

\newcommand\pfitem[1]{\par\noindent(#1):}
\newcommand\pfcase[2]{\smallskip\noindent{\em Case #1: #2}\noindent}
\newcommand\set[1]{\ensuremath{\{#1\}}}
\newcommand\id[1]{\boldsymbol1\set{#1}}
\newcommand\intoi{\int_0^1}
\newcommand\cFN{\cF_N}
\newcommand\cFoo{\cF_\infty}
\newcommand\Goo{G_\infty}
\newcommand\bGoo{\bG_{\infty}}
\newcommand\ioo{_{i=1}^\infty}
\newcommand\ntoo{\ensuremath{{n\to\infty}}}
\newcommand\eqd{\overset{\mathrm{d}}{=}}
\newcommand\E{\operatorname{\mathbb E}}
\newcommand\Prb{\operatorname{\mathbb P}}
\newcommand\Var{\operatorname{Var}}
\newcommand\Be{\operatorname{Be}}
\newcommand\eps{\varepsilon}
\renewcommand\le{\leqslant}
\renewcommand\ge{\geqslant}

\newcommand\bigpar[1]{\bigl(#1\bigr)}

\newcommand\biggpar[1]{\biggl(#1\biggr)}

\newcommand\cD{\mathcal D}
\newcommand\cE{\mathcal E}
\newcommand\cF{\mathcal F}
\newcommand\cH{\mathcal H}
\newcommand\cO{\mathcal O}
\newcommand\cP{\mathcal P}
\newcommand\cV{\mathcal V}
\newcommand\cW{\mathcal W}

\newcommand\R{\mathbb R}
\newcommand\Z{\mathbb Z}

\newcommand\tG{\tilde G}
\newcommand\tU{\tilde U}
\newcommand\bG{\overline G}
\newcommand\bK{\overline K}
\newcommand\bKn{\overline{K}_n}
\newcommand\bP{\overline P}

\newdimen\unit\newdimen\psep\newcount\nd\newcount\ndx\newbox\dotb
\newdimen\dx\newdimen\dy\newdimen\dxx\newdimen\dyy\newdimen\hgt
\newcommand\clap[1]{\hbox to 0pt{\hss{#1}\hss}}
\newcommand\vdisk[1]{{\setbox0\clap{\font\dotf=cmr10 scaled #1\dotf.}%
\raise-.5\ht0\box0}}
\newcommand\vblob[1]{{\setbox0\clap{$#1$}\raise-.55\ht0\box0}}
\newcommand\varline[2]{\setbox\dotb\hbox{\vdisk{#1}}\psep=#2\ht\dotb}
\newcommand\point[3]{\rlap{\kern#1\unit\raise#2\unit\hbox{#3}}}
\newcommand\setnd[4]{\dx=#3\unit\advance\dx-#1\unit\divide\dx by\psep
\dy=#4\unit\advance\dy-#2\unit\divide\dy by\psep
\multiply\dx by\dx\multiply\dy by\dy\advance\dx\dy\nd=1
\loop\ifnum\dx>0\advance\dx-\nd sp\advance\nd1\advance\dx-\nd sp\repeat}
\newcommand\dl[4]{{\setnd{#1}{#2}{#3}{#4}\dline{#1}{#2}{#3}{#4}\nd}}
\newcommand\dline[5]{{\nd=#5\hgt=#2\unit\dx=#3\unit\advance\dx-#1\unit
\divide\dx by\nd\dy=#4\unit\advance\dy-#2\unit\divide\dy by\nd
\rlap{\kern#1\unit\loop\ifnum\nd>1\advance\nd-1\advance\hgt\dy
\kern\dx\raise\hgt\copy\dotb\repeat}}}
\newcommand\ptlr[3]{\point{#1}{#2}{\raise-.4ex\rlap{$\ \,\scriptstyle{#3}$}}}
\newcommand\ptll[3]{\point{#1}{#2}{\raise-.4ex\llap{$\scriptstyle{#3}\ $}}}
\newcommand\pt[2]{\point{#1}{#2}{\Large\vblob{\bullet}}}
\newcommand\pe[2]{\point{#1}{#2}{\Large\vblob{\circ}}}
\varline{1000}{.3}\unit=1em
\begin{document}

\begin{abstract}
Given a hereditary graph property~$\cP$, consider distributions of random
orderings of vertices of graphs $G\in\cP$ that are preserved under
isomorphisms and under taking induced subgraphs. We show that for many
properties $\cP$ the only such random orderings are uniform, and give some
examples of non-uniform orderings when they exist.
\end{abstract}
\maketitle

\section{Introduction}

For any (finite or countably infinite) graph~$G$, write $\cO_G$ for the set of
possible total orderings of the vertex set $V(G)$, and write $\cD_G$ for the
set of all probability distributions on~$\cO_G$. (For countably infinite
graphs, we use the $\sigma$-algebra generated by all events of the form
$u<v$, $u,v\in V(G)$.) Recall that $H$ is an {\em induced\/} subgraph
of $G$ if the vertex set $V(H)$ is a subset of $V(G)$ and an edge $xy$ lies in
$H$ if and only if $x,y\in V(H)$ and $xy$ is an edge of~$G$. Note that an
induced subgraph is determined by the subset $V(H)\subseteq V(G)$. We shall
write $G[S]$ for the induced subgraph of $G$ with vertex set~$S$.

We call a distribution $\Prb_G\in\cD_G$ {\em consistent\/} if for any two
finite isomorphic induced subgraphs $H_1$, $H_2$ and any isomorphism
$\phi\colon H_1\to H_2$, the induced orders on $H_1$ and $H_2$ have
distributions that are mapped to each other by~$\phi$, i.e., for all
$v_1,\dots,v_k\in H_1$,
\[
 \Prb_G(v_1<v_2<\dots<v_k)=\Prb_G(\phi(v_1)<\phi(v_2)<\dots<\phi(v_k)).
\]
(In fact this then implies the same result even for infinite induced
subgraphs.)

\begin{example}\label{x:uniform}
Define the {\em uniform\/} random ordering on $G$ by assigning the vertices
i.i.d.\ uniform $U(0,1)$ random variables $X_v$ and declaring that $v_1<v_2$
if and only if $X_{v_1}<X_{v_2}$. This almost surely gives a total ordering
of~$V(G)$, and the resulting distribution of orderings is clearly consistent.
For a finite graph of order~$n$, the uniform random ordering is just the
natural uniform probability distribution on all $|\cO_G|=n!$ orderings
of~$V(G)$.
\end{example}

There are some cases when the uniform random ordering is the only consistent
random ordering. In this case we shall call the graph $G$ itself
{\em uniform}. As an example, consider a {\em homogeneous\/} graph~$G$, namely
a graph that is either a complete graph or an empty graph. As every induced
subgraph of order $k$ is isomorphic to itself by any permutation, we
must have that the ordering on any $k$ vertices is uniformly chosen from the
$k!$ possible orderings. It thus agrees with the uniform model defined above
on any finite subset of vertices, and hence on the whole graph.
The converse is false in general --- there exist infinite non-homogeneous
graphs which are uniform. Indeed, we shall see many examples below. However,
for finite non-homogeneous graphs there are always non-uniform consistent
random orderings (see for example \refT{bounded} and \refL{almostuniform}
below). Hence for finite graphs, $G$ is uniform if and only if it is
homogeneous.

A {\em graph property\/} $\cP$ is a collection of finite labelled graphs
(typically on vertex sets of the form $[n]=\set{1,\dots,n}$), which is closed
under isomorphism, so if the labelled graph $G$ is isomorphic to $G'$ then
$G\in\cP$ if and only if $G'\in\cP$. A graph property is called
{\em hereditary\/} if whenever $G\in\cP$ and $H$ is an induced subgraph of $G$
then $H\in\cP$. Hereditary properties of graphs have been studied for over two
decades, and there is a huge family of results concerning the structure of
graphs, hypergraphs, and other combinatorial structures having a certain
hereditary property, the number of graphs of order $n$ in a property,
the difficulty of approximation by graphs in the property, etc. For a sample
of results, see \cite{Aleks, Al-Stav1, Al-Stav2, Al-Stav3, BBM1, BBM2,
BBSS1, BBSS2,  BBW1, BBW2, BoTho1, BoTho2, BoTho3, BTW, Pr-Ste1, Pr-Ste2, Pr-Ste3, Saxt}.
There are two obvious ways of defining a hereditary property of graphs. First,
let $\cH$ be a collection of graphs, and write $\cF_\cH$ for the hereditary
property consisting of all finite graphs $G$ that do not contain any induced
subgraph isomorphic to some graph in~$\cH$. We call the graphs in this property
$\cH$-{\em free}. Second, the collection
$\cP_G$ of all finite graphs isomorphic to some induced subgraph
of a (finite or countably infinite) graph $G$ is also a hereditary property.

Given a hereditary property $\cP$, consider probability models that assign to
each graph $G\in\cP$ an element $\Prb_G\in\cD_G$, i.e., a random total
ordering of its vertex set~$V(G)$. We call this model {\em consistent\/} if,
whenever $H,G\in \cP$ and $H$ is isomorphic to an induced subgraph $H'$ of $G$,
by say $\phi\colon H\to H'$, then the random order $\Prb_H$ has the same
distribution as the random order induced on $H'$ by~$\Prb_G$. In other words,
for all $x_1,x_2,\dots,x_k\in V(H)$,
\begin{equation}\label{e:induced}
 \Prb_H(x_1<x_2<\dots<x_k)=\Prb_G(\phi(x_1)<\phi(x_2)<\dots<\phi(x_k)).
\end{equation}
(Note that it follows that each $\Prb_G$ is consistent.)
For any hereditary property $\cP$, the {\em uniform\/} model, defined by
choosing the uniform distribution on all orderings of $V(G)$ for each
$G\in\cP$, is clearly consistent. We call the property $\cP$ {\em uniform\/}
if the only consistent ordering model on $\cP$ is the uniform one. This
terminology is justified by the following observation.

\begin{lemma}\label{l:consistent}
 Let\/ $G$ be a finite or countably infinite graph. Then any consistent random
 ordering on\/ $G$ induces a consistent random ordering model on\/~$\cP_G$.
 Conversely, any consistent random ordering model on\/ $\cP_G$ is induced from
 a unique and consistent random ordering on\/~$G$. In particular, $G$ is
 uniform iff\/ $\cP_G$ is uniform.
\end{lemma}
\begin{proof}
Given a consistent ordering $\Prb_G$ on $G$, we define for each $H\in\cP_G$
the random ordering given by~\eqref{e:induced}, where $\phi\colon H\to H'$ is
any identification of $H$ with an induced subgraph $H'$ of~$G$. The fact that
$\Prb_G$ is consistent implies that the distribution of this ordering is
independent of the choice of~$\phi$, and the collection
$\set{\Prb_H}_{H\in\cP_G}$ is clearly a consistent random ordering model
on~$\cP_G$. Conversely, suppose we have a consistent random ordering model
$\set{\Prb_H}_{H\in\cP_G}$ on~$\cP_G$. Define a random ordering on $G$ so that
for any finite set of vertices $x_1,\dots,x_k\in V(G)$,
\begin{equation}\label{e:induced2}
 \Prb(x_1<x_2<\dots<x_k)=\Prb_H(x_1<x_2<\dots<x_k),
\end{equation}
where $H=G[\set{x_1,\dots,x_k}]$. Consistency of $\set{\Prb_H}_{H\in\cP_G}$
implies that this produces a well defined probability distribution in~$\cD_G$,
which is clearly itself consistent. Moreover, any distribution in $\cD_G$
that induces $\set{\Prb_H}_{H\in\cP_G}$ must satisfy~\eqref{e:induced2}, so
this distribution on $\cO_G$ is unique. The last statement also follows as the
random ordering on $G$ is uniform iff it is uniform when restricted to any
finite subgraph.
\end{proof}

The study of consistent ordering models on families of graphs was started
by Angel, Kechris, and Lyons~\cite{AKL}, who showed that the class of all
graphs is uniform, as well as, for example, the class of $K_n$-free graphs.
In fact they studied not only
graphs, but also hypergraphs and metric spaces, and gave several applications
of their results to uniquely ergodic groups. Russ Lyons suggested to the authors
that they continue the study of consistent ordering models on hereditary
properties of graphs.

The main aim of this paper is to show that for many natural choices of hereditary
property~$\cP$, the only consistent ordering model is uniform, thus greatly
extending the result just mentioned in~\cite{AKL}. In particular
we shall prove the following result in \refS{templates}.

\begin{theorem}\label{t:twin}
 Suppose that\/ $\cP$ is a hereditary property such that for any graph\/
 $G\in\cP$ and any vertex\/ $v\in G$ there exists a graph\/ $G'\in\cP$
 which is obtained from\/ $G$ by replacing\/ $v$ by two twin vertices\/
 $v_1$,~$v_2$ with the same neighbourhoods as\/ $v$ in\/ $G\setminus\set{v}$.
 Suppose also that there exists a graph\/ $G\in\cP$ that is not a disjoint
 union of cliques or a complete multipartite graph. Then\/ $\cP$ is uniform.
\end{theorem}

Recall that vertices $v_1,v_2\in G$ are called {\em twins\/} if the
neighbourhoods of $v_1$ and $v_2$ are the same in $G\setminus\set{v_1,v_2}$.
Twin vertices may be either adjacent or non-adjacent.

\begin{remark}
The hereditary properties satisfying the assumption of \refT{twin}
have an equivalent characterization using the theory of graph limits
(see~\cite{Lovasz}). Each graph limit (or graphon) $W$ defines a hereditary
property $\cP_W$ consisting of all graphs $G$ such that the induced subgraph
density $t_{\mathrm{ind}}(G,W)>0$. Lov\'asz and Szegedy
\cite[Proposition 4.10]{LSz:regularity} have shown that
$\cP$ equals a union $\bigcup_{W\in \cW}\cP_W$ for some set $\cW$ of graph
limits if and only if the first condition in \refT{twin} holds.
\end{remark}

\vspace{5pt}
The next result concerns $\cH$-free graphs introduced earlier: it
follows from \refT{twin}, see \refS{templates}.

\begin{theorem}\label{t:free}
 Suppose that\/ $\cH$ is a set of finite graphs such that either no\/
 $H\in\cH$ contains a pair of adjacent twins, or no\/ $H\in\cH$ contains a
 pair of non-adjacent twins. Suppose also that\/ $\cH$ does not contain the
 path\/ $P_3$ on three vertices, or its complement\/~$\bP_3$.
 Then\/ $\cF_\cH$ is uniform.
\end{theorem}

For example, \refT{free} applies to triangle-free graphs (as a triangle does
not contain a pair of non-adjacent twins), claw-free graphs (the claw
$K_{1,3}$ does not contain adjacent twins), and chordal graphs
($\set{C_4,C_5,C_6,\dots}$-free graphs) as cycles of length at least 4 do not
contain adjacent
twins. However it cannot be applied to, for example, the hereditary property
consisting of all graphs of girth at least~5 ($\set{C_3,C_4}$-free graphs) as
$C_3$ contains a pair of adjacent twins and $C_4$ contains a pair of
non-adjacent twins. We can however deduce that the class of all graphs with
girth at least $g$ is uniform from the following more general result,
proved in \refS{glue}.

\begin{theorem}\label{t:joins}
 Assume\/ $\cP$ is a hereditary property such that for any\/ $G_1,G_2\in\cP$
 and any vertices\/ $v_1\in V(G_1)$, $v_2\in V(G_2)$, the graph obtained from
 the disjoint union\/ $G_1\cup G_2$ by identifying the vertices\/ $v_1$
 and\/ $v_2$ also lies in\/~$\cP$. Then\/ $\cP$ is uniform.
\end{theorem}

\begin{remark}
The condition of \refT{joins} is equivalent to the condition that a graph
$G$ lies in $\cP$ if and only if all its 2-connected induced subgraphs do
(or $\cP$ consists only of the empty graph~$K_1$). Indeed, it
is not hard to see that $\cP$ is also closed under disjoint unions.
In particular, \refT{joins} applies to the class of all bipartite graphs, the
class of all forests, and the class of all planar graphs, thus answering
Question~3.4 of~\cite{AKL}. It also generalises Theorem~5.1 of~\cite{AKL}.
Indeed, it shows that the class of all $\cH$-free graphs is uniform
whenever $\cH$ consists only of 2-connected graphs.
\end{remark}

We actually derive \refT{joins} from the more general, but technical,
\refT{glue} given in \refS{glue}.

Although \refT{joins} applies to the class of all
forests, in the case of hereditary properties of forests we can
say much more. Recall that a {\em leaf\/} is a vertex of degree~1.

\begin{theorem}\label{t:leaves}
 Suppose\/ $\cP$ is a hereditary property of forests and suppose that for
 every non-empty forest\/ $F\in\cP$, at least one of the following holds.
 \begin{romenumerate}
  \item There exists a leaf\/ $u$ of\/ $F$ such that any forest obtained
   from\/ $F$ by replacing\/ $u$ by an arbitrary number of\/
   $($non-adjacent$)$ twins and then adding an arbitrary number of independent
   vertices lies in\/~$\cP$.
 \item There exist two leaves\/ $u_1$, $u_2$ of\/ $F$ adjacent to distinct
   vertices\/ $v_1,v_2\in V(F)$ such that the forest obtained by replacing
   both\/ $u_1$ and\/ $u_2$ by arbitrary numbers of\/ $($non-adjacent$)$ twins
   lies in\/~$\cP$.
 \end{romenumerate}
 Then\/ $\cP$ is uniform.
\end{theorem}

\refT{leaves} too is proved in \refS{glue}.
Note that the conditions of \refT{leaves} imply that either $\cP$ consists
entirely of empty graphs, or $\cP$ contains all graphs of the form
$K_{1,n}\cup\bK_m$. (Consider the case when $F$ is a single edge.)
Indeed, the class $\set{K_{1,n}\cup\bK_m}_{n,m\ge0}$ is an example where
\refT{leaves} applies. By comparison, the class of all induced subgraphs
of stars $K_{1,n}$, $n\ge1$, (i.e., the class of all stars and empty graphs)
is not uniform (see \refE{unionKn} below).

\section{Some non-uniform consistent orderings}\label{s:nonuniform}

Before we prove that many properties $\cP$ are uniform, we first give some
examples of properties and graphs with non-uniform consistent orderings.

\begin{example}\label{x:unionKn}
 Suppose that every graph $G\in\cP$ is a disjoint union of cliques,
 and that some $G\in\cP$ is non-homogeneous. We can
 construct a non-uniform consistent order by first taking a uniform random
 order of the cliques, and then a uniform random order of the vertices
 within each clique. By taking graph complements we can similarly construct
 an example when every $G\in\cP$ is a complete multipartite graph. We take a
 uniform random order of the partite classes, and then a uniform random order
 of the vertices within each partite class.
\end{example}

The following results give constructions of non-uniform consistent orderings
for large classes of graphs and properties. The first construction was
suggested by Leonard Schulman and proved by Angel, Kechris and Lyons~\cite{AKL};
the alternative proof we give below was sketched to us by Lyons.

\begin{theorem}\label{t:bounded}
 Suppose that there exists\/ $\Delta<\infty$ such that for every graph\/
 $G\in\cP$, the maximum degree of\/ $G$ is at most\/~$\Delta$. Then there
 exists a consistent random order model on\/ $\cP$ that is non-uniform on
 any non-homogeneous graph in\/~$\cP$.
\end{theorem}
\begin{proof}
Let $G\in\cP$ be a graph with $n$ vertices. We first show that we can embed
$G$ into Euclidean space $\R^n$ in such a way that the distance between
vertices $x,y\in V(G)$ is $c_0$ if $x$ and $y$ are not adjacent, and
$c_1\ne c_0$ if $x$ and $y$ are adjacent in~$G$. Indeed, let $A=(a_{xy})$ be
the adjacency matrix of~$G$, defined by $a_{xy}=1$ if $xy\in E(G)$ and
$a_{xy}=0$ otherwise. Then $A$ is symmetric and all its eigenvalues
are real and lie
between $-\Delta$ and~$\Delta$. Thus if $\eps<1/\Delta$, the matrix
$I_n+\eps A$ is positive definite, and so there exists a symmetric matrix
$B=(b_{ij})$ such that $B^TB=B^2=I_n+\eps A$. Place each vertex $x\in V(G)$
at the point $p_x=(b_{ix})_{i=1}^n\in\R^n$. Then the distance between
any two distinct vertices $x,y\in V(G)$ is given by
$\|p_x-p_y\|^2=p_x\cdot p_x-2p_x\cdot p_y+p_y\cdot p_y=2-2\eps a_{xy}$.
Thus non-adjacent vertices are at distance $c_0=\sqrt2$ and adjacent vertices
are at distance $c_1=\sqrt{2-2\eps}$.

Now construct a random ordering of the vertices of $G$ by taking a unit vector
$u\in\R^n$ uniformly at random, and setting $x<y$ if $p_x\cdot u<p_y\cdot u$.
This almost surely gives a total ordering on $V(G)$ and it is clear that it is
consistent. Indeed, any induced subgraph $H$ is mapped to a set of points that
is isometric to the set of points produced by the same construction applied
to~$H$. We also note that this ordering is non-uniform on~$G$, provided
that $G$ is not homogeneous. Indeed, any non-homogeneous graph contains
a subgraph isomorphic to either the path $P_3$ or its complement~$\bP_3$, and
so it is enough to show that the ordering is non-uniform on any such subgraph.
On such a subgraph, the ordering is given by a
random projection of a non-equilateral triangle, which it is easy to see is
non-uniform. For example, the probability that a vertex $v$ is in the middle
of the ordering is proportional to the angle at the corresponding vertex of
the triangle.
\end{proof}

\begin{lemma}\label{l:almostuniform}
 Let\/ $G$ be a non-homogeneous graph with\/ $n$ vertices. Then there exists a
 non-uniform consistent random ordering that is uniform on any subset of\/
 $n-1$ vertices. Moreover it can be realised by assigning uniform\/
 $($dependent$)$ random variables\/ $X_v\in[0,1]$ to vertices\/ $v\in V(G)$
 in such a way that any set of\/ $n-1$ variables\/ $X_v$ are independent.
\end{lemma}
\begin{proof}
Fix an $\alpha\in[0,1]$ and a $v_0\in V(G)$ and define a random ordering on
$G$ by giving each vertex $v\ne v_0$ an i.i.d.\ $U(0,1)$ random variable
$X_v\in[0,1]$. Pick an edge $xy$ uniformly at random from $G$ (independently
of the~$X_v$, $v\ne v_0$), and define $X_{v_0}\in[0,1]$ so that
\begin{equation}\label{e:linear}
 \sum_{v\in V(H)}\eps_vX_v\equiv\alpha\bmod 1,
\end{equation}
where $\eps_v=-1$ if $v\in\set{x,y}$ and $\eps_v=1$ otherwise. Note that for any
choice of edge $xy\in E(G)$ this is essentially equivalent to assigning
i.i.d.\ $U(0,1)$ random variables to {\em all\/} vertices and conditioning on
the event that \eqref{e:linear} holds. Hence the resulting distribution is
independent of the choice of~$v_0$, and is uniform on any subset of $n-1$
vertices. Moreover, the overall probability distribution on orderings is
obtained by averaging the distributions for each choice of edge $xy\in E(G)$,
and is therefore invariant under any automorphism of~$G$. Consistency
follows as the distribution is uniform on any proper induced subgraph.

We now show that, for suitable~$\alpha$, this ordering is not uniform on $G$
itself. Let the vertices of $G$ be $\set{1,\dots,n}$ and define
$P_{j_1,\dots,j_r}$ to be the probability that
\begin{equation}\label{e:xord}
 X_{j_1}<X_{j_2}<\dots<X_{j_r}<\min\bigl\{X_s:s\notin\set{j_1,\dots,j_r}\bigr\},
\end{equation}
i.e., that $X_{j_1},\dots, X_{j_r}$ are the smallest $r$ values of the~$X_v$,
and in that order. Define $P^{(x,y)}_{j_1,\dots,j_r}$ to be the probability
that \eqref{e:xord} holds conditioned on the chosen edge being $xy\in E(G)$.
Then
\[
 P_{j_1,\dots,j_r}=\frac{1}{|E(G)|}\sum_{xy\in E(G)}
 P^{(x,y)}_{j_1,\dots,j_r}.
\]

Assume first that $G$ is not regular and label the vertices so that the degree
$d_1$ of vertex 1 is not equal to the degree $d_2$ of vertex~2. Consider
\[
 \delta=P_{1,2}-P_{2,1}=\frac{1}{|E(G)|}\sum_{xy\in E(G)}
 \bigpar{P^{(x,y)}_{1,2}-P^{(x,y)}_{2,1}}.
\]
By symmetry, $P^{(x,y)}_{1,2}=P^{(x,y)}_{2,1}$ unless
$|\set{x,y}\cap\set{1,2}|=1$. Hence, again by symmetry,
letting $d_j'$ be the number of neighbours of $j$ in
$V(G)\setminus\set{1,2}$,
\begin{align*}
 |E(G)|\delta&=d'_1\bigpar{P^{(1,3)}_{1,2}-P^{(1,3)}_{2,1}}
 +d'_2\bigpar{P^{(2,3)}_{1,2}-P^{(2,3)}_{2,1}}\\
 &=(d'_1-d'_2)\bigpar{P^{(1,3)}_{1,2}-P^{(1,3)}_{2,1}}\\
 &=(d_1-d_2)\tfrac{(-1)^n}{(n-1)!}\tbinom{n}{2}B_{n-1}(\alpha),
\end{align*}
where the last line follows from \refL{edgedist} and $B_n(x)$ denotes the
$n$th Bernoulli polynomial. In particular $\delta\ne0$ unless $\alpha$ is one
of the zeros of the polynomial $B_{n-1}(x)$.

Now assume $G$ is regular with vertex degree~$d$. As $G$ is not homogeneous, $n\ge4$
and we can order the vertices so that $\set{1,3}\in E(G)$ but
$\set{2,3}\notin E(G)$. Consider
\[
 \delta'=P_{1,2,3}-P_{2,1,3}=\frac{1}{|E(G)|}\sum_{xy\in E(G)}
 \bigpar{P^{(x,y)}_{1,2,3}-P^{(x,y)}_{2,1,3}}.
\]
Once again by symmetry, $P^{(x,y)}_{1,2,3}=P^{(x,y)}_{2,1,3}$ unless
$|\set{x,y}\cap\set{1,2}|=1$. Hence, again by symmetry,
\begin{align*}
 |E(G)|\delta'&=\bigpar{P^{(1,3)}_{1,2,3}-P^{(1,3)}_{2,1,3}}
 +(d-1)\bigpar{P^{(1,4)}_{1,2,3}-P^{(1,4)}_{2,1,3}}
 +d\bigpar{P^{(2,4)}_{1,2,3}-P^{(2,4)}_{2,1,3}}\\
 &=\bigpar{P^{(1,3)}_{1,2,3}-P^{(1,3)}_{2,1,3}}
 -\bigpar{P^{(1,4)}_{1,2,3}-P^{(1,4)}_{2,1,3}}.
\end{align*}
Now
\[
 P^{(1,3)}_{1,2}=\sum_{i>2}P^{(1,3)}_{1,2,i}
 =P^{(1,3)}_{1,2,3}+(n-3)P^{(1,4)}_{1,2,3},
\]
and similarly for $P^{(1,3)}_{2,1}$. Hence by \refL{edgedist}
(noting that $n\ge4$)
\begin{align*}
 |E(G)|\delta'&=\bigpar{P^{(1,3)}_{1,2}-P^{(1,3)}_{2,1}}
 -(n-2)\bigpar{P^{(1,4)}_{1,2,3}-P^{(1,4)}_{2,1,3}}\\
 &=\tfrac{(-1)^n}{(n-1)!}
 \bigpar{\tbinom{n}{2}-(n-2)(n-3+2H_{n-3})}B_{n-1}(\alpha)
 -\tfrac{(-1)^n}{(n-3)!}B_{n-2}(\alpha),
\end{align*}
where $H_n=1+\frac{1}{2}+\dots+\frac{1}{n}$.
As $B_{n-1}(\alpha)$ and $B_{n-2}(\alpha)$ are linearly independent, this is
non-zero for all but a finite number of $\alpha\in[0,1]$.

Thus in all cases the distribution is non-uniform on $V(G)$ for suitable~$\alpha$.
\end{proof}

\begin{theorem}\label{t:disjoint}
 Suppose\/ $\cP$ is a hereditary property and\/ $H$ is a graph on at least\/
 $2$ vertices such that for every\/ $G\in\cP$, all induced subgraphs of\/ $G$
 isomorphic to\/ $H$ are vertex disjoint. Then there is a consistent random
 ordering model on\/ $\cP$ that is uniform on all graphs\/ $G\in\cP$ without
 an induced subgraph isomorphic to\/ $H$, and is non-uniform on all
 non-homogeneous graphs\/ $G\in\cP$ containing\/ $H$ as an induced subgraph.
\end{theorem}
Note that $H$ itself may be either homogeneous or non-homogeneous.
\begin{proof}
Assume first that $H$ is homogeneous. Fix $\alpha\in[0,1]$ and define the
following random order for each $G\in\cP$. Each vertex $v\in V(G)$ is assigned
an i.i.d.\ $U(0,1)$ random variable~$X_v$, except that if $G$ contains induced
subgraphs $H_1,\dots,H_k$ isomorphic to~$H$, a fixed vertex $v_i$ is chosen
from each~$V(H_i)$, and $X_{v_i}\in[0,1]$ is redefined so that
\begin{equation}\label{e:sumalpha}
 \sum_{v\in V(H_i)}X_v\equiv\alpha\bmod1.
\end{equation}
This is essentially equivalent conditioning on the event that
\eqref{e:sumalpha} occurs for each~$i$. The ordering on $G$ is then obtained
from the ordering of the $X_v$ in~$\R$. Note that the joint distribution of
the~$X_v$, $v\in V(G)$, and hence the distribution on the ordering, is
independent of the choices of the~$v_i$, and hence is symmetric under all
permutations of~$V(H_i)$. Let $G'$ be an induced subgraph of $G$ and assume
$G'$ contains $H_i$ only for $i\in S\subseteq\set{1,\dots,k}$. By independence
of the choice of $v_i$ we may assume $v_i\notin V(G')$ for $i\notin S$. Hence
the induced ordering on $G'$ is given by exactly the same model. By
independence on the~$v_i$, the distribution is clearly invariant under
automorphisms of~$G'$, so the random ordering model described is consistent
on~$\cP$. It is also clearly uniform on any $G\in\cP$ that does not contain
$H$ as an induced subgraph. It remains to show that if $G\in\cP$ does contain
$H$ as a proper induced subgraph then the ordering on $G$ is non-uniform.
(Note that in this case $G$ is necessarily non-homogeneous as otherwise it
would contain non vertex-disjoint copies of~$H$.) Let
$v\in V(G)\setminus V(H)$ and assume $V(H)=\set{1,\dots,n}$. Then by
\refL{addx},
\[
 \Prb(X_v<X_1<\dots<X_n)=\tfrac{1}{(n+1)!}+\tfrac{(-1)^{n-1}}{n!^2}B_n(\alpha).
\]
Hence, for all but a finite number of choices of~$\alpha$, this probability is
not $1/(n+1)!$ as it would be in the case of the uniform distribution. Thus
the distribution is not uniform on $G$ for a suitable choice of~$\alpha$.

Assume now that $H$ itself is not homogeneous. Fix a non-uniform distribution
on $H$ as given by \refL{almostuniform}. Fix $G\in\cP$ and suppose $G$
contains (vertex-disjoint) copies $H_1,\dots,H_k$ of~$H$. Define a random
ordering on the vertices of $G\in\cP$ by giving each vertex $v\in V(G)$ an
independent uniform random variable $X_v\in[0,1]$, except that on each $H_i$
we apply the construction of \refL{almostuniform}, independently for
each~$H_i$. In other words, we fix a choice of vertex $v_i\in V(H_i)$ and
then uniformly and independently choose one edge from each~$H_i$. The random
variable $X_{v_i}$ is then redefined so that \eqref{e:linear} holds on
each~$H_i$. Once again, if $G'$ is an induced subgraph of $G$ containing only
the copies~$H_i$, $i\in S\subseteq\set{1,\dots,k}$, then we can without loss of
generality assume that $v_i\notin V(G')$ for each $i\notin S$. Then the
induced ordering on $G'$ is given by exactly the same model. Hence the
ordering model on $\cP$ is consistent and has the stated properties.
\end{proof}

\begin{remark}\label{r:disjoint}
We note that it is important in \refT{disjoint} that the copies of $H$ be
{\em vertex\/} disjoint. For example, taking $H$ as a single edge and $\cP$ as
any of the uniform properties mentioned above gives examples with each copy of
$H$ being {\em edge}-disjoint but the conclusion of \refT{disjoint} failing.
Another instructive example is given in \refE{2flowers} below, where the
copies of $H$ intersect in at most one vertex and each copy has ``private''
vertices not included in any other copy of~$H$. Nevertheless $\cP$ is
still uniform.
\end{remark}

Despite \refR{disjoint}, a construction similar to that in \refT{disjoint} is
occasionally possible even when not all copies of $H$ are vertex disjoint. The
following gives an example.

\begin{example}\label{x:doublebroom}
Let $n\ge3$ and define $G$ to be the infinite {\em double broom\/} consisting
of a path $P_n$ on $n$ vertices with an infinite number of leaves added to
the end-vertices of~$P_n$ (so that the longest path in $G$ is $P_{n+2}$).
Let the vertices of the central path be $u_1,\dots,u_n$.
Assign i.i.d.\ $U(0,1)$ random variables $X_v$ to all $v\in V(G)$ except
that $X_{u_n}\in[0,1]$ is redefined so that
$\sum_{i=1}^n X_{u_i}\equiv\alpha\bmod1$, where $\alpha\in[0,1]$ is a zero of
the Bernoulli polynomial~$B_n(x)$. Any induced subgraph of $G$ that does not
contain all vertices of the central path $P_n$ receives a uniform ordering, as
does $P_n$ itself (by symmetry). The only remaining induced subgraphs are
$P_{n+1}$, single brooms containing $P_n$ and at least two leaves attached at
one end-vertex, and double brooms with one or more leaves at each end.
Any pair of such single brooms or double brooms are isomorphic only by an
isomorphism which either fixes $P_n$ or reverses its direction, and hence
receive the same distribution of orderings. Any copy of $P_{n+1}$ consists of
the central $P_n$ with one leaf at either end. Such a graph has the uniform
random ordering by \refL{addx} as $B_n(\alpha)=0$. (The $X_{u_i}$ are
exchangeable, so it is enough to check the distribution of the rank of the
leaf $v$ in the ordering of $v,u_1,\dots,u_n$.) Thus the ordering is
consistent. On the other hand, $B_{n+1}(\alpha)\ne0$ by \refL{Bzeros},
so the second formula in \refL{addx} implies that the random ordering is
{\em not\/} uniform on any $P_{n+2}$ subgraph.
\end{example}

\begin{remark}
Note that the ordering in \refE{doublebroom} is not consistent for $n=2$
(the infinite double star) as the single brooms obtained by adding leaves to
one end-vertex of a $P_2$ are in fact stars, and have many automorphisms
which do not preserve the distribution of the given random order. This is to
be expected as the infinite double star is in fact uniform by \refT{leaves}.
Moreover, the class of all induced subgraphs of double brooms with central
path of length $\le n$ is also uniform by \refT{leaves}. \refE{doublebroom}
demonstrates that \refT{leaves} does not however apply when the central path
length is required to be {\em exactly\/}~$n$. Indeed, the single broom
subgraphs of the double brooms do not satisfy the conditions of \refT{leaves}.
\end{remark}

\section{Templates and infinite blow-ups}\label{s:basic}

Consider a (finite) {\em template\/}~$G$, i.e., a graph with a set $\cV$ of
vertices, each vertex labelled as either {\em full\/} or {\em empty}. Define
the infinite blow-up $\Goo$ of $G$ as an infinite graph with vertex set
$\bigcup_{v\in\cV}\cW_v$ where $\cW_v=\set{v_i}\ioo$, such that $\cW_v$ induces
an empty or complete graph according to whether $v$ is empty or full
respectively, and for any distinct $v,w\in \cV$ and all $i,j\ge1$,
$v_iw_j$ is an edge in $\Goo$ if and only if $vw$ is an edge in~$G$.
Define the hereditary property $\cP_G$ as the set of all finite
induced subgraphs of~$\Goo$, i.e., $\cP_G=\cP_{\Goo}$. We shall call a
template $G$ {\em uniform\/} if $\Goo$ (or equivalently~$\cP_G$) is uniform,
i.e., if the only consistent random ordering is the uniform one. Our aim is to
prove that most templates are uniform. This is, however, not always the case.

\begin{example}\label{x:UK}
 Suppose that the template has no edges and at least two vertices with at
 least one of the vertices full. Thus $\Goo$ is a disjoint union of some
 infinite cliques and (perhaps) some infinite empty graphs, and thus a
 disjoint union of at least two cliques (infinite or singletons). Any induced
 subgraph is thus also a disjoint union of cliques. We can construct a
 non-uniform consistent order as in \refE{unionKn} by first taking a uniform
 random order of the cliques, and then a uniform random order of the vertices
 within each clique.
\end{example}

Consider first each `blob' $\cW_v$ separately. Fix $v\in\cV$ and
$v_i\in\cW_v$. Since any permutation of $\cW_v$ is an automorphism of~$\Goo$,
and thus preserves the distribution of the order, the random variables
$\set{\id{v_i>v_k}}_{k\ne i}$ are exchangeable. Thus, by de Finetti's theorem,
see e.g.,~\cite[Theorem 1.1 and Proposition 1.4]{Kallenberg:symm},
a.s.\ there exists a limit
\begin{equation}\label{e:U}
 U_{v_i}:=\lim_\ntoo\frac1n\sum_{k=1}^n\id{v_i>v_k}.
\end{equation}
Thus each $U_{v_i}$ is a random variable with $U_{v_i}\in[0,1]$. Moreover, if
$v_i<v_j$, then $\id{v_i>v_k}\le\id{v_j>v_k}$ for every~$k$, and thus
$U_{v_i}\le U_{v_j}$.

\begin{lemma}\label{l:Uiid}
 For each\/~$v$, $\set{U_{v_i}}\ioo$ is a sequence of i.i.d.\ uniformly
 distributed random variables;\/ $U_{v_i}\sim U(0,1)$.
\end{lemma}
\begin{proof}
The order restricted to $\cW_v$ has a distribution invariant under all
permutations, and thus it is the uniform random order. We may thus assume that
the random order on $\cW_v$ is defined by a collection of i.i.d.\ uniform
random variables $X_{v_i}$ as in \refE{uniform}. But then \eqref{e:U} and
the law of large numbers a.s.\ yield
\begin{equation}
 U_{v_i}=\lim_\ntoo\frac1n\sum_{k=1}^n\id{X_{v_i}>X_{v_k}}=X_{v_i}.
\end{equation}
\end{proof}

Moreover, this extends to all blobs, jointly.

\begin{lemma}\label{l:UUiid}
 The random variables\/ $U_{v_i}$, $v\in\cV$ and\/ $i\ge1$, are i.i.d.\ and
 uniform on\/ $[0,1]$.
\end{lemma}
\begin{proof}
Consider a finite subset $A_v$ of each~$\cW_v$. Any permutation of $A_v$ is
an automorphism of~$\Goo$, and thus the induced order on $A_v$ is the uniform
random order, and this also holds even if we condition on the induced orders
on all $A_w$, $w\ne v$. Hence the induced orders on the subsets $A_v$ are
independent (and uniform). Since the sets $A_v$ are arbitrary finite subsets
of the~$\cW_v$, this means that the induced orders on the sets $\cW_v$,
$v\in\cV$, are independent, and thus the families $\set{U_{v_i}}\ioo$,
$v\in\cV$, are independent.
\end{proof}

Next, take two  vertices $v,u\in\cV$ and compare vertices in the two blobs
$\cW_v$ and~$\cW_u$. For every $v_i\in\cW_v$, we see in analogy
with~\eqref{e:U}, again by de Finetti's theorem, that a.s.\ the limit
\begin{equation}\label{e:V}
 V_{u,v_i}:=\lim_\ntoo\frac1n\sum_{k=1}^n\id{v_i>u_k}
\end{equation}
exists. Note that $V_{v,v_i}=U_{v_i}$. Each $V_{u,v_i}$ is a random variable
with values in $[0,1]$ and gives the `rank' of vertex $v_i$ with respect
to~$\cW_u$, i.e., the proportion of vertices in $\cW_u$ that it exceeds. Note
that these random variables are in general neither independent nor uniform.

\begin{example}\label{x:K+K}
 Let the template consist of two full vertices and no edge; thus
 $\cV=\set{1,2}$ and $\Goo$ consists of two disjoint infinite cliques.
 For the random order described in \refE{UK}, we have
 $V_{1,2_i}=V_{1,2_j}\in\set{0,1}$ for all $i,j\ge1$, and
 $V_{1,2_i}\sim\Be(1/2)$.
\end{example}

\begin{lemma}\label{l:F}
 For each pair\/ $u,v\in\cV$, there exists a random distribution function\/
 $F_{u,v}$ on\/ $[0,1]$ such that, a.s., for every\/ $x\in[0,1]$,
 \begin{equation}\label{e:F}
  F_{u,v}(x)=\lim_\ntoo\frac1n\sum_{i=1}^n\id{V_{u,v_i}\le x}.
 \end{equation}
 Furthermore, conditioned on\/ $F_{u,v}$, the random variables\/ $V_{u,v_i}$,
 $i\ge1$, are i.i.d.\ with cumulative distribution function\/~$F_{u,v}$.
\end{lemma}

\begin{remark}\label{r:u=v}
When $v=u$, this holds by \refL{UUiid} with $F_{u,u}(x)=x$ a.s., so
$F_{u,u}$ is non-random.
\end{remark}

\begin{proof}
We may assume that $v\ne u$. Since any permutation of $\cW_v$ is an
automorphism of~$\Goo$, it follows from \eqref{e:V} that the random variables
$\set{V_{u,v_i}}\ioo$ are exchangeable. The result follows from another application
of de Finetti's theorem.
\end{proof}

It follows immediately from the definition \eqref{e:V} that, for any
$v,w\in\cV$ and $i,j\ge1$,
\begin{equation}\label{e:vw}
 v_i<w_j\implies V_{u,v_i}\le V_{u,w_j}.
\end{equation}
Equivalently, interchanging $v_i$ and~$w_j$,
\begin{equation}\label{e:vw2}
 V_{u,v_i}<V_{u,w_j}\implies v_i<w_j.
\end{equation}

\begin{remark}\label{r:order}
The order is thus described by the variables~$V_{u,v_i}$, for any fixed
$u\in\cV$, in the case when these random variables are a.s.\ distinct.
(This is not necessarily the case, as is seen in \refE{K+K}; in that example
the variables $V_{1,2_i}$ do not identify the order on~$\cW_2$. See also
\refR{Venough} below.)
\end{remark}

Let $\cFN$ be the $\sigma$-field generated by all events $v_i<w_j$ for
$v,w\in\cV$ and $i,j>N$, and let $\cFoo:=\bigcap_{N=0}^\infty\cFN$ be the
tail $\sigma$-field.

\begin{lemma}\label{l:Ftail}
 Each\/ $F_{u,v}$ is $\cFoo$-measurable.
\end{lemma}
\begin{proof}
As the limits \eqref{e:V} and \eqref{e:F} do not depend on the first $N$ terms
in the sums, $V_{u,v_i}$, $i>N$, and hence $F_{u,v}$ are $\cFN$-measurable
for all~$N$.
\end{proof}

\begin{lemma}\label{l:UFindep}
 The i.i.d.\ uniform random variables $U_{v_i}$, $v\in\cV$ and\/ $i\ge1$,
 are\/ $($jointly$)$ independent of\/~$\cFoo$. Thus the two families\/
 $\set{U_{v_i}}_{v,i}$ and\/ $\set{F_{u,v}}_{u,v\in\cV}$ are independent.
\end{lemma}
Note that the random variables $\set{F_{u,v}}_{u,v\in\cV}$ may be dependent
on each other.
\begin{proof}
The induced orders on the subsets $\cW_{v,N}:=\set{v_i}_{i=1}^N$, $v\in\cV$,
are independent and uniform, even conditioned on~$\cFN$, since permutations
of $\cW_{v,N}$ are automorphisms of~$\Goo$. Hence these induced orders are
independent of~$\cFoo$, and letting $N\to\infty$, we obtain that the induced
orders on the blobs $\cW_v$, $v\in\cV$, are (jointly) independent of~$\cFoo$.
The random variables $U_{v_i}$ depend on these induced orders only.
The result now follows by \refL{Ftail}.
\end{proof}

We note some useful formulae.
\begin{lemma}\label{l:FV}
 Let\/ $v,u\in\cV$. Then the following hold a.s.
 \begin{romenumerate}
 \item For every\/ $i\ge1$,
  \begin{equation}\label{e:V=supU}
   V_{u,v_i}=\sup_k\set{U_{u_k}:u_k<v_i}.
  \end{equation}
 \item For every\/ $i\ge1$,
  \begin{equation}\label{e:V=FU}
   V_{u,v_i}=F_{v,u}(U_{v_i}).
  \end{equation}
 \item For\/ $x\in[0,1]$,
  \begin{equation}
   F_{u,v}(x)=\sup\set{s:F_{v,u}(s)\le x}.
  \end{equation}
  Hence, $F_{u,v}$ is the right-continuous inverse of\/ $F_{v,u}$.
 \end{romenumerate}
\end{lemma}
\begin{proof}
\pfitem{i}
Let $x:=\sup_k\set{U_{u_k}:u_k<v_i}$. Then
\[
 U_{u_j}<x \implies u_j<v_i \implies U_{u_j}\le x.
\]
Hence \eqref{e:V=supU} follows from definition \eqref{e:V} and the law of
large numbers.

\pfitem{ii}
By \eqref{e:vw}--\eqref{e:vw2}, recalling that $U_{v_i}=V_{v,v_i}$,
\[
 V_{v,u_k}<U_{v_i} \implies u_k<v_i \implies V_{v,u_k}\le U_{v_i}.
\]
Hence, the definitions \eqref{e:V} and \eqref{e:F} yield, a.s.,
\[
 F_{v,u}(U_{v_i}-)\le V_{u,v_i}\le F_{v,u}(U_{v_i}).
\]
Since $U_{v_i}$ is a continuous random variable, and independent of $F_{u,v}$
by \refL{UFindep}, $U_{v_i}$ is a.s.\ a continuity point of~$F_{u,v}$, and the
result follows.

\pfitem{iii}
By \eqref{e:F}, \eqref{e:V=FU} and the fact that $\set{U_{v_i}}_i$ are
i.i.d.\ and uniform, a.s.,
\[
 F_{u,v}(x)=\lim_\ntoo\frac1n\sum_{i=1}^n\id{F_{v,u}(U_{v_i})\le x}
 =\sup\set{s:F_{v,u}(s)\le x}.
\]
(This holds a.s.\ for e.g., all rational $x\in[0,1]$, and thus for all $x$
simultaneously.)
\end{proof}

\begin{theorem}\label{t:uniform}
 Fix any\/ $u\in\cV$. Then the following are equivalent.
 \begin{romenumerate}
  \item The random order on\/ $\Goo$ is uniform.
  \item The random variables\/ $V_{u,v_i}$, $v\in\cV$ and\/ $i\ge1$, are
   i.i.d.\ and uniform.
  \item The random distribution functions\/ $F_{u,v}$, $u,v\in\cV$, are
   a.s.\ equal to the identity; $F_{u,v}(x)=x$, $x\in[0,1]$.
 \end{romenumerate}
\end{theorem}
We may assume $u\ne v$ in (iii) as this always holds for $u=v$;
see \refR{u=v}.
\begin{proof}
(i)$\implies$(ii):
We may assume that the random order is given by i.i.d.\ uniform random
variables $X_{v_i}$ as in \refE{uniform}, and then $V_{u,v_i}=X_{v_i}$
a.s.\ by \eqref{e:V} and the law of large numbers.

\noindent(ii)$\implies$(i):
Immediate by \eqref{e:vw2}.

\noindent(ii)$\implies$(iii):
By \eqref{e:F} and the law of large numbers.

\noindent(iii)$\implies$(ii):
By \refL{FV}(iii), $F_{v,u}=F_{u,v}^{-1}$ is the identity. Thus \refL{FV}(ii)
yields $V_{u,v_i}=U_{v_i}$ and (ii) follows by \refL{UUiid}.
\end{proof}

\begin{remark}\label{r:Venough}
Consider again any consistent order on~$\Goo$. It follows from Lemmas
\ref{l:UUiid} and \ref{l:UFindep} together with \eqref{e:V=FU} that, for any
pair $u,v\in\cV$ and $i,j\ge1$, the random variables $V_{u,u_i}=U_{u_i}$ and
$V_{u,v_j}=F_{v,u}(U_{v_j})$ are independent, with $U_{u_i}$ uniform. In
particular, these two random variables are a.s.\ distinct, and thus they
determine the order between $u_i$ and $v_j$ by \eqref{e:vw}--\eqref{e:vw2}.
Hence, the order is a.s.\ determined by the collection of all $V_{u,v_i}$
($u,v\in\cV$, $i\ge1$). (As remarked in \refR{order}, it is sometimes, but not
always, possible to use just a single~$u$.)

Note also that Lemmas \ref{l:UUiid} and \ref{l:UFindep} together with
\eqref{e:V=FU} show that the random $V_{u,v_i}$ may be constructed by randomly
selecting first $\set{F_{u,v}}_{u,v}$ with the right distribution and then
i.i.d.\ uniform~$U_{u_i}$, and defining $V_{u,v_i}:=F_{v,u}(U_{v_i})$. The
conditional distribution of $V_{u,v_i}$ given $\set{F_{w,z}}_{w,z\in\cV}$ is
thus $F_{u,v}$, c.f.\ \refL{F}.
\end{remark}

\begin{remark}
This section only uses automorphisms of $\Goo$ that preserves each~$\cW_v$
(and thus is a permutation of each~$\cW_v$). \refR{Venough} thus gives a
description of all random orders that are invariant under this group of
permutations of the vertices of~$\Goo$. (Conversely, the construction
above yields such a random order. In particular, if we fix $u$ and any
distribution of $\set{F_{u,v}}_v$ such that each $F_{u,v}$ is continuous,
this defines a random order of this type on~$\Goo$. If some $F_{u,v}$
have atoms, we may have to further specify the order.)
\end{remark}

\section{Uniformity of templates}\label{s:templates}

Recall that a template $G$ is uniform if the only consistent random order on
$\Goo$ is the uniform random order.

\begin{remark}\label{r:comp}
If $G$ is uniform, then so is its complement $\bG$ (with the labels full and
empty interchanged), since the corresponding graphs $\Goo$ and $\bGoo$ are
complements of each other, and thus have the same isomorphisms between
subgraphs.
\end{remark}

\begin{lemma}\label{l:homog}
 A template with a single vertex is uniform. More generally, any template
 consisting only of empty vertices and no edges is uniform, and so is
 any complete template consisting only of full vertices.
\end{lemma}
\begin{proof}
In the cases described, $\Goo$ is homogeneous, and thus any permutation
of the vertices is an isomorphism. Hence any consistent random order is
uniform. (Cf.\ the proof of \refL{Uiid}.)
\end{proof}

Given a consistent random order on~$\Goo$, we define a relation $\equiv$ on
$\cV$ by letting $v\equiv w$ if the induced random order on $\cW_v\cup\cW_w$
is uniform. This relation is clearly symmetric, and it is reflexive
by \refL{homog}. We shall soon see that it also is transitive.

\begin{lemma}\label{l:=}
 Suppose that\/ $v,w\in\cV$. Then the following are equivalent.
 \begin{romenumerate}
  \item\label{lv=w} $v\equiv w$.
  \item\label{lv=wV} $V_{v,u_i}=V_{w,u_i}$ a.s., for every\/
   $u\in\cV$ and\/ $i\ge1$.
  \item\label{lv=wF} $F_{v,u}=F_{w,u}$ a.s., for every\/ $u\in\cV$.
  \item\label{lv=wFx} $F_{u,v}=F_{u,w}$ a.s., for every\/ $u\in\cV$.
  \item\label{lv=wF2} $F_{w,v}(x)=x$ a.s., for every\/ $x\in[0,1]$.
 \end{romenumerate}
\end{lemma}
\begin{proof}
\ref{lv=w}$\implies$\ref{lv=wV}:
Suppose $v\equiv w$.
By \refT{uniform} applied to $\cW_v\cup\cW_w$, $F_{w,v}(x)=F_{v,w}(x)=x$ a.s.
Hence, \refL{FV}(ii) yields $V_{v,w_i}=U_{w_i}$.

Fix $u$ and~$i$. Let $\eps>0$ and choose first a $j\ge1$ such that
$U_{v_j}\in(V_{v,u_i}-\eps,V_{v,u_i})$ and then a $k\ge1$ such that
$U_{w_k}\in(V_{v,u_i}-\eps,U_{v_j})$. Then
$V_{v,w_k}=U_{w_k}<U_{v_j}<V_{v,u_i}$, so $w_k<u_i$ by~\eqref{e:vw2}.
Hence, \eqref{e:V=supU} yields
\[
 V_{w,u_i}\ge U_{w_k} > V_{v,u_i}-\eps.
\]
Since $\eps$ is arbitrary, this yields $V_{w,u_i}\ge V_{v,u_i}$,
Interchanging $v$ and $w$ we obtain~\ref{lv=wV}.

\noindent\ref{lv=wV}$\implies$\ref{lv=wF}:
By definition \eqref{e:F}.

\noindent\ref{lv=wF}$\implies$\ref{lv=wFx}:
By \refL{FV}(iii).

\noindent\ref{lv=wFx}$\implies$\ref{lv=wF2}:
Taking $u=w$ we have $F_{w,v}(x)=F_{w,w}(x)=x$.

\noindent\ref{lv=wF2}$\implies$\ref{lv=w}:
\refT{uniform} shows that the induced random order on $\cW_v\cup\cW_w$ is
uniform.
\end{proof}

\begin{corollary}\label{c:equiv}
 The relation\/ $\equiv$ is an equivalence relation on\/~$\cV$.
\end{corollary}
\begin{proof}
By \refL{=}, since (for example) \ref{lv=wV} defines an equivalence
relation.
\end{proof}

\begin{corollary}\label{c:V=U}
 If\/ $v\equiv w$, then\/ $V_{w,v_i}=U_{v_i}$ a.s.\ for every\/ $i\ge1$.
\end{corollary}
\begin{proof}
By \refL{=}, $V_{w,v_i}=V_{v,v_i}=U_{v_i}$.
\end{proof}

\begin{lemma}\label{l:=unif}
 The random order on\/ $\Goo$ is uniform if and only if\/ $v\equiv w$ for any
 two vertices\/ $v,w\in\cV$.
\end{lemma}
\begin{proof}
A consequence of \refL{=} and \refT{uniform}.
\end{proof}

\begin{lemma}\label{l:2pairs}
 Suppose that the template\/ $G$ contains two\/ $($not necessarily disjoint$)$
 pairs\/ $u,v$ and\/ $w,z$ such that the induced subtemplates with vertices\/
 $\set{u,v}$ and\/ $\set{w,z}$ are isomorphic. If\/ $u\equiv v$, then\/
 $w\equiv z$.
\end{lemma}
\begin{proof}
The induced subgraphs of $\Goo$ on $\cW_u\cup\cW_v$ and $\cW_w\cup\cW_z$ are
isomorphic, and thus the induced random orders on these subgraphs have
distributions that are mapped to each other by the isomorphism mapping
$u_i\mapsto w_i$ and $v_i\mapsto z_i$. Hence, if the random order induced on
$\cW_u\cup\cW_v$ is uniform, then so is the random order induced on
$\cW_w\cup\cW_z$.
\end{proof}

\begin{lemma}\label{l:allpairs}
 Suppose that the template\/ $G$ contains an induced subtemplate\/ $H$
 such that any consistent ordering on $G_\infty$ induces a uniform
 ordering on~$H_\infty$.
 Furthermore suppose\/ $H$ contains two $($not necessarily
 disjoint$)$ pairs of vertices\/ $u,v$ and\/ $u',v'$ such that\/ $u$ and\/
 $u'$ are full, $v$ and\/ $v'$ are empty, and furthermore\/ $uv\in E(G)$
 and\/ $u'v'\notin E(G)$. Then\/ $G$ is uniform.
\end{lemma}
\begin{proof}
Since the ordering on $H_\infty$ is uniform,
we have $u\equiv v\equiv u'\equiv v'$.

If $z\in\cV$ is empty and $zu\in E(G)$, then the subtemplates \set{z,u} and
\set{v,u} are isomorphic, Since $v\equiv u$, we have $z\equiv u$ by
\refL{2pairs}.

If $z\in\cV$ is empty and $zu\notin E(G)$, we argue similarly using the
isomorphic subtemplates \set{z,u} and \set{v',u'} and obtain
$z\equiv u$.

If $z\in\cV$ is full we argue similarly using the pairs \set{z,v} and
\set{u,v}, or \set{z,v} and \set{u',v'} to obtain $z\equiv v$.

Hence $z\equiv u\equiv v$ for every $z\in\cV$, and \refL{=unif} shows that the random
order on $\Goo$ is uniform.
\end{proof}

We now show that any template $G$ containing certain 3-vertex subtemplates
are necessarily uniform (see \refF{subtemplates}).

\begin{figure}
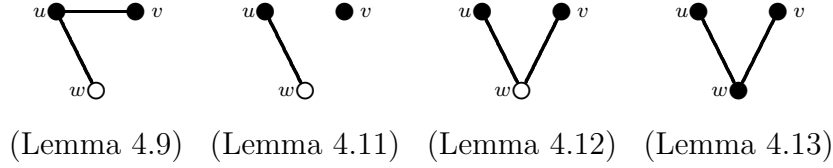

\[\unit=15pt
 \begin{array}{cccc}
 \ptll{0}{0}{w}\ptll{-1}{2}{u}\ptlr{1}{2}{v}
 \pe{0}{0}\pt{-1}{2}\pt{1}{2}\dl{-.1}{.2}{-1}{2}\dl{-1}{2}{1}{2}&
 \ptll{0}{0}{w}\ptll{-1}{2}{u}\ptlr{1}{2}{v}
 \pe{0}{0}\pt{-1}{2}\pt{1}{2}\dl{-.1}{.2}{-1}{2}&
 \ptll{0}{0}{w}\ptll{-1}{2}{u}\ptlr{1}{2}{v}
 \pe{0}{0}\pt{-1}{2}\pt{1}{2}\dl{-.1}{.2}{-1}{2}\dl{.1}{.2}{1}{2}&
 \ptll{0}{0}{w}\ptll{-1}{2}{u}\ptlr{1}{2}{v}
 \pt{0}{0}\pt{-1}{2}\pt{1}{2}\dl{-.1}{.2}{-1}{2}\dl{.1}{.2}{1}{2}\\[10pt]
 \text{(\refL{e-f-f})}&\text{(\refL{e-f,f})}&
 \text{(\refL{f-e-f})}&\text{(\refL{f-f-f})}
 \end{array}
\]
\caption{Subtemplates implying uniformity of $G$.}\label{f:subtemplates}
\end{figure}

\begin{lemma}\label{l:e-f-f}
 Suppose that the template\/ $G$ contains two full vertices\/ $u$ and\/ $v$
 and an empty vertex\/~$w$, with\/ $uv,uw\in E(G)$ and\/ $vw\notin E(G)$.
 Then\/ $G$ is uniform.
\end{lemma}
\begin{proof}
First, $u\equiv v$ by \refL{homog} applied to the subtemplate \set{u,v}.

The two subgraphs induced by $\cW_u\cup\set{w_1,w_2}$ and
$\cW_u\cup\set{w_1,v_1}$ are isomorphic, by an isomorphism mapping
$w_2\to v_1$ and fixing everything else; thus the distributions of their
induced random orders are mapped to each other by this isomorphism.
Hence, by \eqref{e:V} and \refC{V=U},
\[
 (V_{u,w_1},V_{u,w_2})\eqd(V_{u,w_1},V_{u,v_1})=(V_{u,w_1},U_{v_1}).
\]
Let $x,y\in[0,1]$. By \refL{F},
$\Prb(V_{u,w_1}\le x,\,V_{u,w_2}\le y)=\E(F_{u,w}(x)F_{u,w}(y))$.
Similarly, and also using \refL{UFindep},
$\Prb(V_{u,w_1}\le x,\,U_{v_1}\le y)=\E (F_{u,w}(x))y$. Hence,
\begin{equation}\label{e:eff}
 \E\bigpar{F_{u,w}(x)F_{u,w}(y)}=\E(F_{u,w}(x))y,\qquad x,y\in[0,1].
\end{equation}
Taking $x=1$ in \eqref{e:eff} yields $\E F_{u,w}(y)=y$, and then taking
$x=y$ yields
\[
 \E\bigpar{F_{u,w}(x)^2}=\E(F_{u,w}(x))x=\bigpar{\E F_{u,w}(x)}^2.
\]
Hence $\Var(F_{u,w}(x))=0$, and thus $F_{u,w}(x)=\E F_{u,w}(x)=x$ a.s.
Consequently, $w\equiv u$ by \refL{=}.

We have shown that $w\equiv u\equiv v$. In other words, the ordering
is uniform on the subtemplate induced by \set{u,v,w}.
The result follows from \refL{allpairs},
using the pairs $u,w$ and~$v,w$.
\end{proof}

\begin{lemma}\label{l:FFi}
 Let\/ $F\colon[0,1]\to[0,1]$ be a distribution function on\/ $[0,1]$, and
 let\/ $F^{-1}\colon[0,1]\to[0,1]$ be its right-continuous inverse. If\/ $X$
 and\/ $Y$ are random variables such that\/ $X$ has distribution\/ $F$ and\/
 $Y$ has distribution\/~$F^{-1}$, then
 \[
  \E(X^2)+\E(Y^2)\ge\tfrac23,
 \]
 with equality if and only if\/ $F$ is the uniform distribution\/ $F(x)=x$.
\end{lemma}
\begin{proof}
Note first the well-known formula
\[
 \E X^2=\E\intoi 2x\,\id{x<X}\,dx=\intoi 2x(1-F(x))\,dx.
\]
Next, if $U\sim U(0,1)$, then $F(U)$ has the distribution function~$F^{-1}$,
so $Y\eqd F(U)$ and thus
\[
 \E Y^2=\E F(U)^2=\intoi F(x)^2\,dx.
\]
Hence,
\begin{align*}
 \E X^2+\E Y^2&=\intoi\bigpar{2x(1-F(x))+F(x)^2}\,dx\\
 &=\intoi(F(x)-x)^2\,dx+\intoi(2x-x^2)\,dx\\
 &=\intoi(F(x)-x)^2\,dx+\frac{2}{3}.
\end{align*}
The result follows.
\end{proof}

\begin{lemma}\label{l:e-f,f}
 Suppose that the template\/ $G$ contains two full vertices\/ $u$ and\/ $v$
 and an empty vertex\/~$w$, with\/ $uw\in E(G)$, $uv,vw\notin E(G)$.
 Then\/ $G$ is uniform.
\end{lemma}
\begin{proof}
Let $\cW_u':=\cW_u\setminus\set{u_1}$. There is an isomorphism between
$\cW_u'\cup\set{w_1}\cup\cW_v$ and $\cW_u\cup\cW_v$ fixing $\cW_u'\cup\cW_v$
and sending $w_1$ to~$u_1$.
It follows that $V_{u,w_1}\eqd U_{u_1}$, even when conditioned on the order in
$\cW_u'\cup\cW_v$. Since $F_{v,u}$ is determined by the order in
$\cW_u'\cup\cW_v$, it follows for any $x\in[0,1]$, using also \refL{UFindep}
and \refR{Venough}, that
\[
 \E(F_{u,w}(x)\mid F_{v,u})=\Prb(V_{u,w_1}\le x\mid F_{v,u})
 =\Prb(U_{u_1}\le x\mid F_{v,u})=\Prb(U_{u_1}\le x)=x.
\]
Since $V_{u,v_1}=F_{v,u}(U_{v_1})$ by \eqref{e:V=FU}, and $U_{v_1}$ is
independent of \set{F_{u,w},F_{u,v}}, it follows that
\begin{equation}\label{e:sw}
 \E(F_{u,w}(V_{u,v_1})\mid F_{v,u},U_{v_1})
 =\E(F_{u,w}(F_{v,u}(U_{v_1})))\mid F_{v,u},U_{v_1})
 =F_{v,u}(U_{v_1})=V_{u,v_1}.
\end{equation}
Next, note that by the same isomorphism,
$\Prb(V_{u,v_1}=V_{u,w_1})=\Prb(V_{u,v_1}=U_{u_1})=0$, since $U_{u_1}$ is
continuous and independent of~$V_{u,v_1}$.
By symmetry, a.s.\
$V_{u,v_1}\ne V_{u,w_j}$ for every~$j$, and thus by \refR{Venough},
these random variables determine the order between $v_1$ and~$w_j$.
It follows that, a.s., using~\eqref{e:F},
\begin{equation}\label{e:ma}
 F_{u,w}(V_{u,v_1})=\lim_\ntoo\frac1n\sum_{i=1}^n\id{V_{u,w_i}\le V_{u,v_1}}
 =\lim_\ntoo\frac1n\sum_{i=1}^n\id{w_i< v_1}.
\end{equation}
However, $\cW_w\cup\set{v_1}$ is an infinite empty graph, isomorphic to~$\cW_w$,
and by \refL{Uiid} and \eqref{e:U}, the r.h.s.\ has a uniform distribution.
Thus $\tU:=F_{u,w}(V_{u,v_1})\sim U(0,1)$, and by \eqref{e:sw},
$\E(\tU\mid F_{v,u},U_{v_1})=V_{u,v_1}$. Consequently,
\begin{equation}\label{e:var}
 \tfrac13=\E(\tU^2)=\E(\tU-V_{u,v_1})^2+\E V_{u,v_1}^2\ge\E V_{u,v_1}^2.
\end{equation}
By the obvious isomorphism of $\cW_u\cup\cW_v$ interchanging $\cW_u$
and~$\cW_v$, $V_{v,u_1}\eqd V_{u,v_1}$, so $\E V_{v,u_1}^2\le\frac13$ too.

Conditioned on $F_{u,v}$ and $F_{v,u}=F_{u,v}^{-1}$, $V_{u,v_1}$ and
$V_{v,u_1}$ have distributions $F_{u,v}$ and $F_{v,u}$, and thus \refL{FFi}
applies and yields
\begin{equation}\label{e:varp}
 \E(V_{u,v_1}^2+V_{v,u_1}^2\mid F_{u,v})\ge\tfrac23.
\end{equation}
Thus, taking the expectation,
\begin{equation}\label{e:var2}
 \E(V_{u,v_1}^2+V_{v,u_1}^2)\ge\tfrac23.
\end{equation}

Consequently, there must be equality in both \eqref{e:var} and~\eqref{e:var2},
and thus a.s.\ in~\eqref{e:varp}. By \refL{FFi}, this implies that
$F_{v,u}(x)=F_{u,v}(x)=x$ a.s. Furthermore, by~\eqref{e:var},
$F_{u,w}(V_{u,v_1})=\tU=V_{u,v_1}$ a.s., where
$V_{u,v_1}=F_{v,u}(U_{v_1})=U_{v_1}$ is independent of~$F_{u,w}$, and thus
$F_{u,w}(x)=x$. Thus $v\equiv u\equiv w$ by \refL{=}.

This shows that the ordering is uniform on the subgraph of $G$ induced
by \set{u,v,w}. Finally, $G$ is uniform by \refL{allpairs} applied to the pairs $u,w$
and~$v,w$.
\end{proof}

\begin{lemma}\label{l:f-e-f}
 Suppose that the template\/ $G$ contains two full vertices\/ $u$ and\/~$v$,
 and one empty vertex\/~$w$, with\/ $uw,vw\in E(G)$ and\/ $uv\notin E(G)$.
 Then\/ $G$ is uniform.
\end{lemma}
\begin{proof}
The induced subgraph of $\Goo$ with vertex set $\set{w_1,w_2,u_1,v_1}$ has an
isomorphism $w_1\leftrightarrow u_1$, $w_2\leftrightarrow v_1$. Hence, the
assumption that the random order of $\Goo$ is consistent implies
\begin{equation}\label{e:c0}
 \Prb(w_1,w_2<u_1)=\Prb(u_1,v_1<w_1).
\end{equation}
By \eqref{e:vw}--\eqref{e:vw2}, and since $\set{V_{u,w_i}}_i$ are independent
of $U_{u_1}$ by \refL{UFindep} and \eqref{e:V=FU} (or \refR{Venough}),
\begin{equation}\label{e:c1}
 \Prb(w_1,w_2<u_1)=\Prb(V_{u,w_1},V_{u,w_2}<U_{u_1})
 =\E(F_{u,w}(U_{u_1})^2)=\E\intoi F_{u,w}(x)^2\,dx.
\end{equation}
Furthermore, $\set{w_1}\cup\cW_u$ is an infinite complete graph, and thus
$V_{u,w_1}\eqd U_{u_i}\sim U(0,1)$, see \refL{Uiid}. Thus, for $x\in[0,1]$,
\[
 x=\Prb(V_{u,w_1}\le x)=\E F_{u,w}(x).
\]
Consequently, by \eqref{e:c1} and the Cauchy--Schwarz inequality,
\begin{equation}\label{e:c2}
 \Prb(w_1,w_2<u_1)=\intoi\E(F_{u,w}(x)^2)\,dx\ge
 \intoi(\E F_{u,w}(x))^2\,dx=\intoi x^2\,dx=\tfrac13.
\end{equation}
On the other hand, again using the Cauchy--Schwarz inequality,
\begin{align}
 \Prb(u_1,v_1<w_1)&=\Prb(V_{w,u_1},V_{w,v_1}<U_{w_1})
 =\E(F_{w,u}(U_{w_1})F_{w,v}(U_{w_1}))\notag\\
 &\le\bigpar{\E(F_{w,u}(U_{w_1})^2)}^{1/2}
 \bigpar{\E(F_{w,v}(U_{w_1})^2)}^{1/2}.\label{e:c3}
\end{align}
By \eqref{e:V=FU}, $F_{w,u}(U_{w_1})=V_{u,w_1}$ and, as noted above,
$V_{u,w_1}\sim U(0,1)$. Hence we deduce that $\E(F_{w,u}(U_{w_1})^2)=\frac13$.
Similarly, by symmetry, $\E(F_{w,v}(U_{w_1})^2)=\frac13$. Consequently,
\eqref{e:c3} yields
\begin{equation}\label{e:c4}
 \Prb(u_1,v_1<w_1)\le\tfrac13.
\end{equation}
By \eqref{e:c0}, we thus must have equality in both \eqref{e:c2}
and~\eqref{e:c3}. The equality in \eqref{e:c2} implies that for a.e.~$x$,
$F_{u,w}(x)=\E F_{u,w}(x)=x$ a.s., which implies that a.s.\ $F_{u,w}(x)=x$ for
all $x\in[0,1]$. Hence $w\equiv u$ by \refL{=}.
By symmetry, $w\equiv v$ also.

Suppose $z$ is any full vertex of~$G$. If $uz\in E(G)$ then $z\equiv u$
by \refL{homog}, while if $uz\notin E(G)$ then $z\equiv u$ by applying
\refL{2pairs} to $\set{u,v}$ and $\set{u,z}$. Now suppose $z$ is an empty
vertex of~$G$. If $zu\in E(G)$ then $z\equiv w\equiv u$ by applying
\refL{2pairs} to $\set{u,w}$ and $\set{u,z}$. If $zw\notin E(G)$ then
$z\equiv w\equiv u$ by applying \refL{homog} to $\set{z,w}$. Finally, if
$zw\in E(G)$ and $zu\notin E(G)$ then we deduce that $\bG$, and hence~$G$, is
uniform by applying \refL{e-f,f} to \set{z,u,w} in~$\bG$. In all cases we see
that $z\equiv u$. Hence $G$ is uniform by \refL{=unif}.
\end{proof}

\begin{lemma}\label{l:f-f-f}
 Suppose that the subgraph of\/ $G$ induced by the set of full vertices has
 a component that is not a clique. Then\/ $G$ is uniform.
\end{lemma}
\begin{proof}
The assumption implies that there exist three full vertices $u$, $v$, $w$
in $G$ with $uw,vw\in E(G)$, but $uv\notin E(G)$. By \refL{homog}, $w\equiv u$
and $w\equiv v$. For any full vertex $z\ne v,w$, either the template induced
by \set{w,z} is isomorphic to that induced by \set{u,v} or that induced by
\set{u,w}. Hence $z\equiv w$ for every full vertex~$z$.

Now let $z$ be an empty vertex. We consider two cases.

\pfcase{1}{Either\/ $uz\in E(G)$ or\/ $vz\in E(G)$.}
In this case $G$ is uniform by either \refL{e-f,f} or \refL{f-e-f} applied to
the subtemplate \set{u,v,z}.

\pfcase{2}{$uz,vz\notin E(G)$.}
Let $\cW_z':=\cW_z\setminus\set{z_1,z_2}$. In this case, the subgraphs of
$\Goo$ induced by $\cW_z$ and $\cW_z'\cup\set{u_1,v_1}$ are isomorphic, by an
isomorphism fixing~$\cW_z'$. Again, using that the random order is consistent,
it follows by \eqref{e:V} that
\[
 (V_{z,u_1},V_{z,v_1})\eqd(V_{z,z_1},V_{z,z_2})=(U_{z_1},U_{z_2}).
\]
Hence, arguing as in the proof of \refL{e-f-f}, for $x\in[0,1]$,
$\E F_{z,u}(x)=\Prb(V_{z,u_1}\le x)=\Prb(U_{z_1}\le x)=x$ and
\[
 \E(F_{z,u}(x)F_{z,v}(x))=\Prb(V_{z,u_1}\le x,\,V_{z,v_1}\le x)
 =\Prb(U_{z_1}\le x,\,U_{z_2}\le x)=x^2.
\]
Furthermore, $F_{z,u}=F_{z,v}$ a.s., by \refL{=} since $u\equiv v$.
Consequently, we have $\E(F_{z,u}(x)^2)=(\E F_{z,u}(x))^2$, and thus
$F_{z,u}(x)=\E F_{z,u}(x)=x$ a.s. Hence $z\equiv u$ by \refL{=}.

In both cases we see that $z\equiv u$. Hence $G$ is uniform by \refL{=unif}.
\end{proof}

We call a template $G$ {\em reduced\/} if if contains no adjacent twin full
vertices, and no non-adjacent twin empty vertices. Clearly any adjacent twin
full vertices or non-adjacent twin empty vertices can be merged in an
non-reduced template $G$ without affecting $\Goo$ and hence without affecting
whether or not $G$ is uniform. Merging all such twins results in a reduced
template, so it is enough to consider just these.

\begin{theorem}\label{t:template}
 If\/ $G$ is a non-uniform reduced template, then\/ $G$ is either an empty
 graph\/ $($with at most one empty vertex$)$ or complete\/ $($with at most one
 full vertex$)$. In particular, for any non-uniform template\/~$G$, $\Goo$
 is either a disjoint union of cliques or a complete multipartite graph.
\end{theorem}
\begin{proof}
By Lemmas~\ref{l:e-f-f} and \ref{l:e-f,f}, any empty vertex must be joined to
either all the full vertices, or none of them. By taking complements we also
have that each full vertex is either joined to all empty vertices or none of
them. Thus either all full vertices are joined to all empty vertices, or no
full vertex is joined to any empty vertex. Without loss of generality (taking
complements if necessary), we may assume that every full vertex is joined
to every empty vertex.

By \refL{f-f-f}, the subgraph of $G$ induced by the full vertices consists of
a disjoint union of cliques. Since we assume $G$ is reduced and any two full
vertices
in a clique of full vertices would be adjacent twins, we deduce that no two
full vertices are adjacent. Similarly, applying \refL{f-f-f} to the complement
of~$G$, we may assume any two empty vertices are adjacent.

If $G$ contained at least two full vertices and at least one empty vertices,
then $G$ would be uniform by \refL{f-e-f}. Hence we deduce that either there
is no empty vertex, and $G$ is an empty graph of full vertices; or there is at
most one full vertex and $G$ is a complete graph consisting of empty vertices
and at most one full vertex.
\end{proof}

\begin{lemma}\label{l:3}
Suppose that\/ $G$ is a template and that\/ $\Goo$ has a consistent random
order such that for any three vertices\/ $u,v,w\in V(\Goo)$, the induced
random ordering on\/ \set{u,v,w} is uniform. Then the ordering is uniform.
\end{lemma}
\begin{proof}
Pick any two vertices $u,v\in V(G)$, and consider the three vertices
$u_1,u_2,v_1$ in $\Goo$. By \refR{Venough} (and the argument there),
$U_{u_1}$, $U_{u_2}$ and $V_{u,v_1}=F_{v,u}(U_{v_1})$ are independent, with
$U_{u_i}$ uniform, and these three random variables
determine the order between $u_1$, $u_2$ and $v_1$.
By assumption, this order is uniform, and thus
\[
 \tfrac13=\Prb(u_1,u_2<v_1)=
\Prb(U_{u_1}, U_{u_2} <V_{u,v_1})
 =\E(V_{u,v_1}^2).
\]
Similarly,  $\E(V_{v,u_1}^2)=\frac13$.
As in the proof of \refL{e-f,f}, it follows from \refL{FFi} that
$F_{u,v}(x)=x$ a.s., and thus $u\equiv v$ by \refL{=}.
As $u$ and $v$ were arbitrary, the ordering on
$\Goo$ is uniform by \refL{=unif}.
\end{proof}

\begin{proof}[Proof of \refT{twin}]
Consider a consistent ordering model on $\cP$.

Suppose $G\in\cP$.
By repeatedly replacing vertices by twins and using Ramsey's theorem on each
subgraph corresponding to one of the original vertices of~$G$, we see
that for all $N>0$ there exists a $G_N\in\cP$ which is obtained from
$G$ by replacing each vertex with either a complete graph or an empty graph
on $N$ vertices. By the infinite pigeonhole principle, there must be a
template $G'$ with underlying graph $G$ such that for infinitely many~$N$,
$G_N$ is an induced subgraph of $\Goo'$ (with $N$ copies of each vertex in $G$).
But then $\cP_{G'}\subseteq\bigcup_{N=1}^\infty\cP_{G_N}\subseteq\cP$.
Hence the random ordering model on $\cP$ induces a random ordering model
on~$\cP_{G'}$.

Suppose first that $G$ is not a disjoint union of cliques or a complete
multipartite graph. Since $\Goo'$ contains $G$ as an induced subgraph,
\refT{template} shows that the template $G'$ is uniform.
In particular, the random ordering on $G\in\cP_{G'}$ is uniform.

As $G$ is not a disjoint union of cliques,
it contains an induced subgraph isomorphic
to the path $P_3$ on three vertices. Similarly, as $G$ is not complete
multipartite, $G$ contains the graph $\overline{P}_3$ consisting of an edge
and an isolated vertex. Thus $P_3,\overline{P}_3\in\cP$ and receive the
uniform ordering on their vertices. The only other graphs on three vertices
are homogeneous, so we deduce that for any graph $H\in\cP$ and any three
vertices $u,v,w\in V(H)$, the induced random ordering on \set{u,v,w} is
uniform.

Now suppose $G$ is any graph in~$\cP$. Let, as above, $G'$ be a template
with underlying graph $G$ and $\cP_{G'}\subseteq\cP$. By what we just have
shown, any set of three vertices in $\Goo'$ receives the uniform ordering,
and thus the ordering of $\Goo'$ is uniform by \refL{3}. Hence the ordering
of $G$ is uniform.
\end{proof}

\begin{proof}[Proof of \refT{free}.]
The hereditary property $\cF_\cH$ has the property that for any $G\in\cF_\cH$
and $v\in V(G)$, some graph $G'$ obtained by replacing $v$ by twins $v_1,v_2$
is also in~$\cF_\cH$. Indeed, we can take the twins to be adjacent if there
is no graph $H\in\cH$ with adjacent twins, and we can take $v_1,v_2$ to be
non-adjacent if there is no graph $H\in\cH$ with non-adjacent twins. In both
cases no copy of $H\in\cH$ in $G'$ could use both vertices $v_1,v_2$, and
hence $H$ would have to be an induced subgraph of~$G$. Without loss of
generality (by taking complements if necessary), assume we are in the first
case, so that any vertex can be replaced by adjacent twins. If $P_3\in\cF_\cH$
then we are done by \refT{twin} as $\cF_\cH$ contains blowups of $P_3$ that
are neither a disjoint union of cliques nor complete multipartite (for example, a
triangle with a pendant edge). If $P_3\notin\cF_\cH$, then $\cH$ must contain
an induced subgraph of~$P_3$. As $P_3\notin\cH$, $\cH$ must then contain a graph
with two (or fewer) vertices. But then $\cF_\cH$ consists only of homogeneous
graphs, and is therefore uniform.
\end{proof}

\section{Gluing graphs}\label{s:glue}

In this section we show in particular that
hereditary properties that are closed
under joining graphs at a single vertex, and many hereditary properties of
forests, are uniform. We start by proving the result for any hereditary
property that satisfies a certain technical condition.

Denote the disjoint union of two graphs $G_1$ and $G_2$ by $G_1\cup G_2$.
Suppose $G$ is a graph and $H$ is an induced subgraph. Define the graph
$[G]^n_H$ to be the graph obtained by taking $n$ copies of $G$ (i.e.,
$G\cup G\cup\dots\cup G$, $n$ times) and identifying the corresponding
subgraphs $H$ from each copy. Thus, for example,
$|V([G]^n_H)|=n|V(G)\setminus V(H)|+|V(H)|$. Let $\bKn$ denote the empty
graph on $n$ vertices. We also extend these notations in the obvious
way to the case when $n=\infty$.

\begin{theorem}\label{t:glue}
 Suppose\/ $\cP$ is a hereditary property such that for any\/
 $G\in\cP$ with at least\/ $2$ vertices, there exists a proper induced
 subgraph\/ $H\ne\emptyset$ of\/ $G$ such that for all\/ $n\ge1$,
 $[G]^n_H\cup[G]^n_H\cup\overline{K}_n\in\cP$. Then\/ $\cP$ is uniform.
\end{theorem}
\begin{proof}
We may assume $\cP$ contains some non-empty graph as otherwise $\cP$
is clearly uniform. Note that, by taking an induced subgraph,
for any $G\in\cP$,
$G\cup G\cup\bK_n\in\cP$. (For $|V(G)|<2$ take an induced subgraph of
$G'\cup G'\cup\bK_n$ with $|V(G')|\ge2$.) We shall prove by induction on
$|V(G)|$ that if $G\in\cP$ then the ordering on $G\cup G\cup\bK_n$ is uniform
for any~$n$. This clearly implies the result. As $G\cup G\cup\bK_n$ is
homogeneous for $|V(G)|<2$, we may assume $|V(G)|\ge2$. Thus by assumption
there exists a proper induced subgraph $H\ne\emptyset$ of $G$ such that for
all $n\ge1$, $[G]^n_H\cup[G]^n_H\cup\bK_n\in\cP$.
Let $\tG=[G]^\infty_H\cup[G]^\infty_H\cup\bK_\infty$. Then
$\cP_{\tG}\subseteq\cP$, and so the consistent ordering on $\cP$ induces an
consistent ordering on $\cP_{\tG}$, and hence on~$\tG$
(see \refL{consistent}). Denote the vertices
of $\bK_\infty$ as $\set{u_i}\ioo$, and the copies of $H$ as $H_i$, $i=1,2$,
with vertices $V(H_i)=\set{v_{i,1},\dots,v_{i,r}}$. Denote the remaining
vertices in the $j$th copy of $G':=G\setminus H$ associated to $H_i$
as~$w_{i,j,k}$, $k=1,\dots,s$. Let $\tG'=\tG\setminus(H_1\cup H_2)$ be
the graph $\tG$ with the two copies of $H$ removed, so that $\tG'$ consists of
an infinite number of disjoint copies of $G'$ together with~$\bK_\infty$.
We first consider the induced random ordering on~$\tG'$. One can define random
variables
\[
 V_{i,j,k}=V_{u,w_{i,j,k}}:=
 \lim_\ntoo\frac1n\sum_{\ell=1}^n\id{w_{i,j,k}>u_\ell}
\]
as in \refS{basic} giving the order of the $w_{i,j,k}$ relative to the vertices
in the $\bK_\infty$ subgraph. As the copies of $G'$ can be permuted in $\tG'$,
the random variables
$V_{i,j}:=(V_{i,j,1},\dots,V_{i,j,s})$ are exchangeable for $i\in\set{1,2}$,
$j\ge1$. Hence, de Finetti's theorem implies that
there is a random distribution $\mu$ on $[0,1]^s$ such that,
conditioned on~$\mu$, the $V_{i,j}$ are i.i.d.\ with distribution~$\mu$.
However, we know that the joint distribution of $V_{i,1}$ and $V_{i,2}$, say,
is uniform as by induction the induced subgraph $G'\cup G'\cup\bK_n$ has a
uniform random order for all~$n$, and hence $G'\cup G'\cup\bK_\infty$ receives
a uniform random ordering. Thus for any measurable subset $S\subseteq [0,1]^s$,
$\E(\mu(S)\mu(S))=|S|^2=\E(\mu(S))\E(\mu(S))$. Thus $\mu$ is a.s.\ constant
and uniform. Thus all $V_{i,j}$ are i.i.d.\ uniform random variables
in $[0,1]^s$, i.e., all $V_{i,j,k}$ are i.i.d.\ $U(0,1)$ random variables.

Let $\cE$ be any event determined by the ordering on
$H_1\cup H_2\cup\bK_\infty$, and assume $\Prb(\cE)=p>0$.
The pairs $(V_{1,j},V_{2,j})$, $j\ge1$, are exchangeable, even conditioned
on~$\cE$. Hence, there is a random measure $\mu_\cE$ on $[0,1]^{2s}$ such that
conditioned on $\cE$ and $\mu_\cE$, $(V_{1,j},V_{2,j})$ are i.i.d.\ with
distribution~$\mu_\cE$.
However, for any measurable subset $S\subseteq[0,1]^{2s}$,
a.s.\ on $\cE$,
\[
 \mu_\cE(S)
 = \lim_\ntoo\frac1n\sum_{j=1}^n\id{(V_{1,j},V_{2,j})\in S}
 =|S|.
\]
Hence, the pairs $(V_{1,j},V_{2,j})$, $j\ge1$, are i.i.d.\ and uniform
even conditioned on~$\cE$. In other words, all $V_{i,j,k}$ are i.i.d.\ and
uniform, and independent of
the ordering on $H_1\cup H_2\cup\bK_\infty$. However, by induction,
the random ordering on $H_1\cup H_2\cup\bK_\infty$ is also uniform
as it is uniform on every subgraph $H_1\cup H_2\cup\bK_n$. The ordering
on $\tG$ is a.s.\ determined by the ordering on $H_1\cup H_2\cup\bK_\infty$
and the variables $V_{i,j,k}$ as the $V_{i,j,k}$ are continuous.
Clearly this distribution is uniform. The result follows as
$G\cup G\cup K_n$ is an induced subgraph of~$\tG$.
\end{proof}

\begin{example}\label{x:flowers}
We note that the requirement that we have two copies of $[G]_H^n$ in
\refT{glue} is essential. For example, let $\cP$ be the set of all graphs that
are induced subgraphs of some $[C_4]^n_{\set{u}}$ (i.e., a collection of
4-cycles with a single vertex identified). Fix $\alpha\in[0,1]$ and assign to
each vertex an i.i.d.\ $U(0,1)$ random variable $X_v$, conditioned so that the
sum of $X_v$ round any 4-cycle is $\alpha\bmod1$. It is not hard to see that
for suitable $\alpha$ this gives a consistent random ordering on $\cP$ which is
not uniform (use \refL{addx}). However for any $G\in\cP$, $|V(G)|\ge2$, the
graph $[G]^n_{H}\cup \bK_n$ lies in $\cP$ for some $H\subset G$, $H\ne\emptyset$.
\end{example}

\begin{example}\label{x:2flowers}
In contrast to \refE{flowers}, let $\cP'$ be the set of all graphs
that are disjoint unions of induced subgraphs of some $[C_4]^n_{\set{u}}$.
Then $\cP'$ satisfies the conditions of \refT{glue}.
Hence $\cP'$ is uniform. Note that the ordering described in \refE{flowers}
is not consistent on $\cP'$ due to the fact that there are two distinct induced
distributions on subgraphs isomorphic to $P_3\cup P_3$ (the one on
$[C_4]^2_{\set{u}}\setminus\set{u}$ not being uniform). The class $\cP'$
also has the property that all $C_4$ subgraphs are edge disjoint, and indeed
also have private vertices that do not belong to any other~$C_4$,
cf.\ \refR{disjoint}.
\end{example}

\begin{proof}[Proof of \refT{joins}.]
If $\cP$ consists only of empty graphs then it is uniform and we are done,
so assume $\cP$ contains some non-empty graph. Then $K_2\in\cP$, and so by
assumption on~$\cP$, $P_3\in\cP$. Take any graph $G\in\cP$ and any vertex
$v\in V(G)$. We can attach multiple copies of $G$ together at $v$ to obtain
$[G]^n_{\set{v}}\in\cP$. Joining two of these to the end-vertices of a $P_3$
and then removing the central vertex gives
$[G]^n_{\set{v}}\cup[G]^n_{\set{v}}\in\cP$. Now repeatedly attaching this
graph to an end-vertex of $P_3$ and removing the central vertex of the $P_3$
gives $[G]^n_{\set{v}}\cup[G]^n_{\set{v}}\cup\bK_n\in\cP$. Hence $\cP$
satisfies the conditions of \refT{glue}, so is uniform.
\end{proof}

In the case when $G\setminus H$ always
is a set of isolated vertices, one can
weaken the conditions of \refT{glue} so that only one copy of $[G]^n_H$
is required. Indeed, in this case we can prove by induction that $G\cup K_n$
is uniform and, in the proof, note that $\tG'$ is an empty graph, so is
automatically uniform. This implies \refT{leaves} in the case when (i)
always holds as we can take $H$ to be $G\setminus\set{u}$. We modify the
proof slightly to obtain \refT{leaves} in its entirety.

\begin{proof}[Proof of \refT{leaves}.]
Given any forest~$F$, write $S_F$ for the set of vertices of $F$ that are
adjacent to a leaf of~$F$. Write $F^*_n$ for the forest obtained
by adding (or deleting) isolated vertices so that $F^*_n$ has exactly
$n$ isolated vertices. For $u\in S_F$, write $F^u_n$ for the forest
obtained by adding (or deleting) leaves attached to $u$ so that
$F^u_n$ has exactly $n$ leaves attached to~$u$.

Consider a consistent random ordering on~$\cP$.
We prove that for every forest $F\in\cP$ and every $u\in S_F\cup\{*\}$,
the random ordering on $F^u_n$ is uniform, provided these graphs lie
in $\cP$ for every~$n$.
The proof is by induction on $|V(F)|$. If $F$ is empty then
$S_F=\emptyset$ and $F^*_n$ is empty, so uniform. Thus we may assume $F$ is
non-empty. As no $F^u_n$ is empty, either (i) or (ii) holds for~$F^u_n$.
This implies there exists $v_n\in S_{F^u_n}\cup\{*\}$ with $v_n\ne u$,
such that the graph $F^{u,v_n}_{n,m}:=(F^u_n)^{v_n}_m$ lies in~$\cP$.
As $S_{F^u_m}=S_F$ is finite, this implies that there is a single
$v\in S_F\cup\{*\}$, $v\ne u$, such that $F^{u,v}_{n,m}\in\cP$
for all $n,m$. Let $F^{u,v}_{\infty,\infty}$ be the infinite
graph with infinitely many leaves or isolated vertices associated
with $u$ and~$v$. Let the leaves or isolated vertices associated
to $u$ be $\{u_i\}_{i\ge1}$ and let the leaves or isolated vertices
associated to $v$ be $\{v_i\}_{i\ge1}$.

Any finite subgraph of $F^{u,v}_{\infty,\infty}$ belongs to~$\cP$,
so the ordering on $\cP$ induces a random ordering on
$F^{u,v}_{\infty,\infty}$.
Assume first that $v\ne*$. As in \refS{basic} we can define random variables
\[
 V_i=V_{u,v_i}:=\lim_\ntoo\frac1n\sum_{k=1}^n\id{v_i>u_k}.
\]
As $\set{u_i}_{i\ge1}\cup\set{v_i}_{i\ge1}$ is a homogeneous set, $V_i$
are i.i.d.\ $U(0,1)$ random variables. Moreover, as in the proof of
\refT{glue}, these random variables are independent of the ordering
on $F^{u,v}_{\infty,0}$. But the random ordering on
$F^{u,v}_{\infty,0}$ is also uniform as it is uniform on all
subgraphs $F^{u,v}_{n,0}$ by induction applied to the proper subgraph
$F^{u,v}_{1,0}$ (or $F^{u,v}_{0,0}$ if $u=*$) of~$F$.
Also, a.s.\ the ordering on $F^{u,v}_{\infty,\infty}$ is determined
by the ordering on $F^{u,v}_{\infty,0}$ and the $V_i$ as the $V_i$
are continuous, and this random ordering is clearly uniform.
If $v=*$ then, interchanging $u$ and $v$, we again have that the
ordering on $F^{u,v}_{\infty,\infty}$ is uniform. Hence in both cases
the ordering on $F^u_n$ is uniform for all~$n$.

Finally we note that for any non-empty $F\in\cP$ conditions (i) or (ii)
imply that there is a $u\in S_F$ such that $F^u_n\in\cP$ for all~$n$.
Hence the ordering on $F$ is also uniform.
\end{proof}

\section{Acknowledgements}

The second author is grateful to the Theory Group at Microsoft Research,
Redmond, and especially to Professor Yuval Peres, whose long-term
hospitality he enjoyed in the spring of 2013, when the first two authors
started this research. The second author is grateful to Trinity College
Cambridge for supporting his research there.
The third author is grateful to Cambridge University and Churchill College,
which he visited when this research was continued.
The authors are also grateful to
Professor Russ Lyons for drawing their attention to his results with
Kechris and Lyons in \cite{AKL}, and presenting several arguments.

\appendix
\section{Non-uniformity of some explicit distributions}

We recall the {\em Bernoulli polynomials} $B_n(x)$, which can be defined by
the generating function
\[
 \frac{te^{xt}}{e^t-1}=\sum_{n=0}^\infty B_n(x)\frac{t^n}{n!},
\]
see e.g.\ \cite[\S24.2]{NIST}.
The first few values are $B_0(x)=1$, $B_1(x)=x-\frac12$,
$B_2(x)=x^2-x+\frac16$, and $B_3(x)=x^3-\frac32x^2+\frac12x$. The most
important property for our purposes is the Fourier series representation
of~$B_n(x)$ \cite[(24.8.3)]{NIST},
\begin{equation}\label{e:bernoulli}
 B_n(x)=-\frac{n!}{(2\pi i)^n}\sum_{k\ne0}\frac{1}{k^n}e^{2\pi ikx},
\end{equation}
which is valid for $x\in[0,1]$ when $n\ge2$ and for $x\in(0,1)$ when $n=1$.

\begin{lemma}\label{l:addx}
 Let\/ $n\ge2$ and\/ $\alpha\in[0,1]$,
 and let\/ $X_1,\dots,X_{n-1}$ and\/ $X$, $X'$ be
 i.i.d.\ $U(0,1)$ random variables. Define\/ $X_n\in[0,1]$ so that
 \begin{equation}\label{e:addx}
  \sum_{i=1}^nX_i\equiv\alpha\bmod 1.
 \end{equation}
 Then for\/ $1\le k\le n$,
 \[
  \Prb\big[X<X_k\text{ and }X_1<X_2<\dots<X_n\big]
  =\frac{k}{(n+1)!}+\frac{(-1)^{n-k}}{n!^2}\binom{n-1}{k-1}B_n(\alpha),
 \]
 and
 \begin{align*}
  \Prb\big[X,X'<X_k&\text{ and }X_1<X_2<\dots<X_n\big]=\\
  &\frac{k(k+1)}{(n+2)!}+\frac{(-1)^{n-k}}{n!(n+1)!}\binom{n-1}{k-1}
  \bigpar{(n+1)B_n(\alpha)+2H_nB_{n+1}(\alpha)},
 \end{align*}
 where $H_n=1+\frac12+\frac13+\dots+\frac1n$.
\end{lemma}
\begin{proof}
Let $P^1_k(\alpha)=\Prb\big[X<X_k\text{ and }X_1<X_2<\dots<X_n]$
and $P^2_k(\alpha)=\Prb\big[X,X'<X_k\text{ and }X_1<X_2<\dots<X_n]$.
If $\alpha$ is replaced by a uniform random variable on $[0,1]$, independent
of $X_1,\dots,X_{n-1},X,X'$, then $X_1,\dots,X_n,X,X'$ are i.i.d.\ $U(0,1)$ random
variables and $\alpha$ satisfies~\eqref{e:addx}.
Thus the Fourier transform
\[
 \hat P^j_k(t):=\int_0^1P^j_k(\alpha)e^{2\pi it\alpha}\,d\alpha,
 \qquad t\in\Z,
\]
can be represented as
\[
 \hat P^j_k(t)=\E\bigpar{P^j_k(\alpha)e^{\omega X_1+\dots+\omega X_n}}
 =\int_{X_1<\dots<X_n} X_k^je^{\omega X_1+\dots+\omega X_n}\,dX_1\dotsm\, dX_n,
\]
where $\omega=2\pi it$. If $t=0$ then $\hat P^1_k(0)=k/(n+1)!$ and
$P^2_k(0)=k(k+1)/(n+2)!$ as there are $k$ (respectively $k(k+1)$) orderings of
$X,X_1,\dots,X_n$ (respectively $X,X',X_1,\dots,X_n$) contributing to $P^j$
and the $X,X',X_1,\dots,X_n$ are i.i.d. Hence we may now assume $t\ne0$.
By symmetry,
\begin{align*}
\hat P^j_k(t)
&=\frac{1}{(k-1)!\,(n-k)!}
\int_{X_1,\dots,X_{k-1}<X_k<X_{k+1},\dots,X_n}
\!X_k^je^{\omega X_1+\dots+\omega X_n}\,dX_1\dotsm\, dX_n
\\
&=\frac{1}{(k-1)!\,(n-k)!}\int_0^1
\biggpar{\int_0^x e^{\omega y}\,dy}^{k-1}
\biggpar{\int_x^1 e^{\omega y}\,dy}^{n-k}
x^je^{\omega x}\,dx
\\
&=\frac{1}{(k-1)!\,(n-k)!\,\omega^{n-1}}\int_0^1
\bigpar{e^{\omega x}-1}^{k-1}
\bigpar{1- e^{\omega x}}^{n-k}
x^je^{\omega x}\,dx
\\
&=\frac{(-1)^{n-k}}{\omega^{n-1}(k-1)!\,(n-k)!}\int_0^1
x^j\bigpar{e^{\omega x}-1}^{n-1}
e^{\omega x}\,dx.
\end{align*}
Integrating by parts gives
\begin{align*}
 \hat P^j_k(t)&=\frac{(-1)^{n-k}}{\omega^nn(k-1)!(n-k)!}
 \bigg(x^j(e^{\omega x}-1)^n\big|_0^1
 -\intoi jx^{j-1}(e^{\omega x}-1)^n\,dx\bigg)\\
 &=\frac{(-1)^{k+1}}{\omega^n n!}\binom{n-1}{k-1}
 \intoi jx^{j-1}(1-e^{\omega x})^n\,dx.
\end{align*}
For $j=1$, expand $(1-e^{\omega x})^n$ using the binomial theorem and
note that $\intoi e^{s\omega x}\,dx=0$ for $s\in\Z\setminus\set{0}$.
This gives
\[
 \hat P^1_k(t)=\frac{(-1)^{k+1}}{\omega^n n!}\binom{n-1}{k-1}.
\]
For $j=2$, we note that
\begin{equation}\label{e:In}
 I_n:=\intoi x(1-e^{\omega x})^n\,dx=\frac{1}{2}-\frac{1}{\omega}H_n,
\end{equation}
where $H_n=1+\frac12+\frac13+\dots+\frac1n$. Indeed, $I_0=\frac12$ and,
for $n\ge1$,
\[
 I_n-I_{n-1}=\intoi\!\! x(-e^{\omega x})(1-e^{\omega x})^{n-1}dx
 =\tfrac{1}{n\omega}x(1-e^{\omega x})^n\Big|_0^1
 -\intoi\!\!\tfrac{1}{n\omega}(1-e^{\omega x})^n dx
 =-\tfrac{1}{n\omega}.
\]
Hence
\[
 \hat P^2_k(t)=\frac{(-1)^{k+1}}{\omega^n n!}\binom{n-1}{k-1}
 \Big(1-\frac{2}{\omega}H_n\Big).
\]
Now we take inverse Fourier transforms, noting that by \eqref{e:bernoulli}
the inverse Fourier transform of $\omega^{-n}$ is
\begin{equation}\label{e:ift}
 \sum_{t\ne0} \frac{1}{\omega^n}e^{-2\pi i \alpha t}
 =\frac{1}{(-2\pi i)^n}\sum_{t\ne0}\frac{1}{(-t)^n}e^{2\pi i\alpha(-t)}
 =-\frac{(-1)^n}{n!}B_n(\alpha).
\end{equation}
We obtain
\[
 P^1_k(\alpha)=\frac{k}{(n+1)!}
 +\sum_{t\ne0}\frac{(-1)^{k+1}}{\omega^nn!}\binom{n-1}{k-1}e^{-2\pi it\alpha}
 =\frac{k}{(n+1)!}+\frac{(-1)^{n-k}}{n!^2}\binom{n-1}{k-1}B_n(\alpha),
\]
and
\begin{align*}
 P^2_k(\alpha)&=\frac{k(k+1)}{(n+2)!}
 +\sum_{t\ne0}\frac{(-1)^{k+1}}{\omega^nn!}\binom{n-1}{k-1}
 \Big(1-\frac{2}{\omega}H_n\Big)e^{-2\pi it\alpha}\\
 &=\frac{k(k+1)}{(n+2)!}+\frac{(-1)^{n-k}}{n!(n+1)!}\binom{n-1}{k-1}
 \bigpar{(n+1)B_n(\alpha)+2H_nB_{n+1}(\alpha)}
\end{align*}
for almost all $\alpha\in[0,1]$. As in both cases both sides are continuous
in~$\alpha$, these in fact hold for all $\alpha\in[0,1]$.
\end{proof}

\begin{lemma}\label{l:edgedist}
 Let\/ $X_1,\dots,X_{n-1}$ be i.i.d.\ $U(0,1)$ random variables. Fix\/
 $\alpha\in[0,1]$ and\/ $1\le i<\ell\le n$ and define\/ $X_n\in[0,1]$ so that
 \begin{equation}\label{e:edgedist}
  \sum_{i\ne i,\ell}^nX_i-X_i-X_\ell\equiv\alpha\bmod 1.
 \end{equation}
 Define $P^{(i,\ell)}_{j_1,j_2,\dots,j_r}$ to be the probability that
 \begin{equation}\label{e:D}
  X_{j_1}<X_{j_2}<\dots<X_{j_r}<\min(X_s:s\notin\set{j_1,\dots,j_r}),
 \end{equation}
 i.e., that the smallest\/ $r$ values of\/ $X_k$ are\/ $X_{j_1},\dots,X_{j_r}$
 in that order. Then for distinct\/ $i,j,k,\ell$,
 \begin{align*}
  P^{(i,\ell)}_{i,j}-P^{(i,\ell)}_{j,i}&
   =\frac{(-1)^n}{(n-1)!}\binom{n}{2}B_{n-1}(\alpha)&(n\ge3)\\
  P^{(i,\ell)}_{i,j,k}-P^{(i,\ell)}_{j,i,k}&=\frac{(-1)^n}{(n-1)!}
   (n-3+2H_{n-3})B_{n-1}(\alpha)+\frac{(-1)^n}{(n-2)!}B_{n-2}(\alpha)&(n\ge4)
 \end{align*}
 where $H_n=1+\frac12+\frac13+\dots+\frac1n$.
\end{lemma}
\begin{proof}
Consider the Fourier transform
\[
 \hat P^{(i,\ell)}_{j_1,\dots,j_r}(t)=\intoi
  P^{(i,\ell)}_{j_1,\dots,j_r}(\alpha)e^{2\pi it\alpha}\,d\alpha,
  \qquad t\in\Z.
\]
If we consider $\alpha$ to be a uniform random variable in $[0,1]$
independent of $X_1,\dots,X_{n-1}$, then $X_1,\dots,X_n$ are now i.i.d.\
$U(0,1)$ random variables and $\alpha$ satisfies~\eqref{e:edgedist}. Thus
\[
 \hat P^{(i,\ell)}_{j_1,\dots,j_r}(t)=\int_D
 e^{\eps_1\omega X_1+\dots+\eps_n\omega X_n}\,dX_1\dotsm dX_n,
\]
where $\omega=2\pi it$, $\eps_s=1$ if $s\ne i,\ell$ and $\eps_i=\eps_\ell=-1$,
and $D$ is the domain given by~\eqref{e:D}.
For the first statement we can by symmetry assume $(i,j,\ell)=(1,2,3)$. Then
\begin{align*}
 \hat P^{(1,3)}_{1,2}(t)&-\hat P^{(1,3)}_{2,1}(t)\\
 &=\hat P^{(1,3)}_{1,2}(t)-\hat P^{(2,3)}_{1,2}(t)\\
 &=\int_{X_1<X_2<X_3,\dots,X_n}\!\!\!
 \bigpar{e^{-\omega X_1+\omega X_2}-e^{\omega X_1-\omega X_2}}
 e^{-\omega X_3+\omega X_4+\dots+\omega X_n}\,dX_1\dotsm dX_n.
\end{align*}
For $t=0$ (i.e., $\omega=0$) this is clearly zero, so assume now that $t\ne0$.
Then
\[
 \int_{X_2}^1 e^{\eps\omega x}\,dx=\frac{1}{\eps\omega}(1-e^{\eps\omega X_2})
\]
for $\eps\in\set{-1,1}$, and
\[
\int_0^{X_2}\bigpar{e^{-\omega X_1+\omega X_2}-e^{\omega X_1-\omega X_2}}\,dX_1
 =\frac{1}{\omega}(e^{\omega X_2}+e^{-\omega X_2}-2).
\]
Hence integrating over all $X_s$, $s\ne2$ gives
\begin{align*}
 \hat P^{(1,3)}_{1,2}(t)-\hat P^{(1,3)}_{2,1}(t)
 &=\frac{1}{\omega^{n-1}}\intoi (e^{\omega x}+e^{-\omega x}-2)
 (-1+e^{-\omega x})(1-e^{\omega x})^{n-3}\,dx\\
 &=\frac{1}{\omega^{n-1}}\intoi \bigpar{e^{\omega x}-1}^2
 \bigpar{1-e^{\omega x}}^{n-2}e^{-2\omega x}\,dx\\
 &=\frac{1}{\omega^{n-1}}\intoi
 \bigpar{1-e^{\omega x}}^{n}e^{-2\omega x}\,dx\\
 &=\frac{1}{\omega^{n-1}}\binom{n}{2},
\end{align*}
where in the last line we have expanded $(1-e^{\omega x})^n$ using the
Binomial Theorem and used that
$\intoi e^{s\omega x}\,dx=0$ for $s\in\Z\setminus\set{0}$. Now take the inverse
Fourier transform using \eqref{e:ift} to give
\[
 P^{(1,3)}_{1,2}(\alpha)-P^{(1,3)}_{2,1}(\alpha)
 =\frac{(-1)^n}{(n-1)!}\binom{n}{2}B_{n-1}(\alpha)
\]
for almost all $\alpha\in[0,1]$. However, as both sides are continuous
in~$\alpha$, this holds for all $\alpha\in[0,1]$.

For the second statement we can assume without loss of generality that
$(i,j,k,\ell)=(1,2,3,4)$. Then, performing the integration over $X_1$,
$X_4,\dots,X_n$, and finally over~$X_2$, we have
\begin{align*}
 \hat P^{(1,4)}_{1,2,3}(t)&-\hat P^{(1,4)}_{2,1,3}(t)\\
 &=\hat P^{(1,4)}_{1,2,3}(t)-\hat P^{(2,4)}_{1,2,3}(t)\\
 &=\int_{X_1<X_2<X_3<X_4,\dots,X_n}\!\!\!
 (e^{-\omega X_1+\omega X_2}-e^{\omega X_1-\omega X_2})
 e^{\omega X_3}e^{-\omega X_4+\omega X_5+\dots}\,dX_1\dotsm dX_n\\
 &=\frac{1}{\omega^{n-2}}
 \int_{X_2<X_3}(e^{\omega X_2}+e^{-\omega X_2}-2)e^{\omega X_3}
 (-1+e^{-\omega X_3})(1-e^{\omega X_3})^{n-4}\,dX_2\,dX_3\\
 &=\frac{1}{\omega^{n-2}}
 \int_{X_2<X_3}(e^{\omega X_2}+e^{-\omega X_2}-2)
 (1-e^{\omega X_3})^{n-3}\,dX_2\,dX_3\\
 &=\frac{1}{\omega^{n-1}}\intoi (e^{\omega x}-e^{-\omega x}-2\omega x)
 (1-e^{\omega x})^{n-3}\,dx.
\end{align*}
Hence, using \eqref{e:In},
\begin{align*}
 \hat P^{(1,4)}_{1,2,3}(t)-\hat P^{(1,4)}_{2,1,3}(t)
 &=\frac{1}{\omega^{n-1}}\intoi (e^{\omega x}-e^{-\omega x}-2\omega x)
 (1-e^{\omega x})^{n-3}\,dx\\
 &=\frac{1}{\omega^{n-1}}\bigpar{(n-3)-\omega+2H_{n-3}}.
\end{align*}
Taking inverse Fourier transforms, again using \eqref{e:ift}, gives
\[
 P^{(1,4)}_{1,2,3}(\alpha)-P^{(1,4)}_{2,1,3}(\alpha)
 =\frac{(-1)^n}{(n-1)!}(n-3+2H_{n-3})B_{n-1}(\alpha)
 +\frac{(-1)^n}{(n-2)!}B_{n-2}(\alpha)
\]
for almost all $\alpha\in[0,1]$, and hence for all $\alpha\in[0,1]$ by
continuity.
\end{proof}

Finally, we record a well-known fact, easily shown by induction
using symmetry and $B_n'(x) =nB_{n-1}(x)$.
\begin{lemma}\label{l:Bzeros}
 The only zeros of\/ $B_n(x)$ in $[0,1]$ are $0,\frac12,1$ for odd\/ $n\ge 3$,
 and exactly two values, one in $(0,\frac12)$ and one in $(\frac12,1)$, for
 even $n\ge2$.\qed
\end{lemma}
In particular, $B_n(x)$ and $B_{n+1}(x)$ have no common zeros in $[0,1]$.
(In fact, this extends to all complex zeros; equivalently, all zeros are
simple, see \cite{Brillhart} and \cite{Dilcher}.)

\end{document}